\documentclass[11pt]{amsart}
\usepackage{amsmath,amsthm,amsfonts,amssymb,amscd}
\usepackage[dvips]{graphicx}
\usepackage{psfrag}
\usepackage{a4wide}

\theoremstyle{plain}
\newtheorem{thm}{Theorem}[section]

\newtheorem{lemma}[thm]{Lemma}
\newtheorem{clly}[thm]{Corollary}
\newtheorem{maintheorem}{Theorem}

\theoremstyle{definition}
\newtheorem{remark}[thm]{Remark}

\newtheorem{defi}[thm]{Definition}

\newtheorem{claim}{Claim}

\newcommand{\RR}{{\mathbb R}}
\newcommand{\NN}{{\mathbb N}}

\renewcommand{\epsilon}{\varepsilon}

\newcommand{\diam}{\operatorname{diam}}

\newcommand{\domi}{\operatorname{Dom}}

\newcommand{\cC}{\EuScript{C}}
\newcommand{\cT}{\EuScript{T}}

\newcommand{\cH}{\EuScript{H}}

\newcommand{\cO}{\EuScript{O}}

\def \NN {{\mathbb N}}

\def \RR {{\mathbb R}}

\def \cC {{\mathcal C}}

\def \cH {{\mathcal H}}

\def \cO {{\mathcal O}}

\def \cS {{\mathcal S}}
\def \cT {{\mathcal T}}

\def \cX {{\mathcal X}}

\def \Hom {{\mathcal  Hom}}
\def \SQ {{\mathcal  SQ}}

\begin{document}

\title[Continuum-wise expansivity and entropy for flows]
{ Continuum-wise  expansivity and entropy for flows}

\author{
Alexander Arbieto, Welington Cordeiro
and Maria Jos\'e Pac\'ifico 
 }
 \thanks{A. A and M.J.P. were partially supported by CNPq,
  PRONEX-Dyn.Syst., FAPERJ. W. C. were partially supported by CAPES.}

\maketitle

\begin{abstract}{We define the concept of continuum wise expansive for flows, and we prove that continuum wise expansive flows on compact metric spaces with topological dimension greater than one have positive entropy.   }
\end{abstract}


\section{Introduction}\label{sec:intro} 

The notion of expansiveness was introduced in the middle of the twentieth century by Utz \cite{Ut}. Expansiveness is a property shared by a large class of dynamical systems exhibiting chaotic behavior. Roughly speaking a system is expansive if two orbits cannot remain close to each other under the action of the system. This notion  is responsible for  many chaotic properties for homeomorphisms defined on compact spaces, see for instance \cite{Hi,Hi2,Le,Ft,Vi} and references therein,  and there is an extensive literature concerning expansive systems. 
A classical result establishes that  every hyperbolic $f$-invariant subset $M$ is expansive. There are many variants of the definition of expansiveness, all of them of interest, as \emph{positive expansiveness} in \cite{Sc}, \emph{point-wise expansiveness} in \cite{Re}, \emph{entropy-expansiveness} in \cite{B3}.
In \cite{Ma} Ma\~n\'e  proved that if a compact metric space
admits an expansive homeomorphism then its topological dimension
is finite. 
In the 90s, Kato introduced the notion of \emph{continuum-wise expansiveness} for homeomorphisms \cite{K1}, and extended the 
result of Ma\~n\'e for $cw$-expansive homeomorphisms.

For  flows, a seminal work is \cite{BW}, where it is  analyzed 
this concept for flows and it is proved that some properties valid for discrete dynamics are also valid for flows. But, the definition of expansiveness
in \cite{BW} does not admit flows with singularities or equilibrium points. 
Using this definition, Keynes and Sears \cite{KS} extend the results 
of Ma\~n\'e  for expansive flows. They proved that if a compact metric
space admits an expansive flow then its topological dimension is finite.
They also proved that expansive flows on manifolds with topological
dimension greater than $1$ have positive entropy.

The goal of this work is to give  a  definition of \emph{continuum-wise expansiveness} for flows, $cw$-expansiveness for short, and extend some of the results already established for discrete dynamics for $cw$-expansive flows. Our main result establishes  that $cw$-expansive  continuos flows on compact metric spaces with
topological dimension greater than $1$ have positive entropy, extending to continuos dynamics the result of Keynes and Sears.

 Throughout this paper, $M$ denotes a compact metric space. To announce precisely our result, let us recall some concepts and definitions already 
 established.
 
A \emph{flow} in $M$ is a family of homemorphisms $\{X^t\}_{t\in\mathbb{R}}$ satisfying $X^0(x)=x$, for all $x\in M$ and $X^{t+s}(x)=X^t(X^s(x))$ for all $s,\, t \in \RR$ and $x\in M$. A \emph{continuum} is a compact connected set and it is \emph{non-degenerate} if it contains more than one point. We denote by $\cC(M)$ the set of all continuum subsets of $M$.

Let $\Hom(\mathbb{R},0)$ be the set of homeomorphisms on $\RR$
fixing the origin and if $A$ is a subset of $M$, $C^0(A,\RR)$ denotes
the set of real continuos maps defined on $A$. Define

\begin{itemize}
\item $\cH(A)=\{\alpha:A\rightarrow \Hom(\mathbb{R},0); \, \exists\,\, x_\alpha\in A \, \mbox{with} \  \alpha(x_\alpha)=id_\mathbb{R} \ \mbox{and}  \ \alpha(.)(t)\in C^0(A,\mathbb{R}), \forall \,\, t\in\mathbb{R} \}$;
\item If $t\in\mathbb{R}$ and $\alpha\in \cH(A)$, $\mathcal{X}^t_\alpha(A)=\{X^{\alpha(x)(t)}(x); x\in A\}$.
\end{itemize}

\begin{figure}[htb]
\begin{center}
\psfrag{a}{$A$}
\psfrag{b}{$X^t(A)$}
\psfrag{c}{$\cX^t(A)$}
\psfrag{d}{$y$}
\psfrag{e}{$x$}
\psfrag{f}{$X^t(y)$}
\psfrag{g}{$X^t(x)$}
\psfrag{h}{$X^{\alpha(y,t)}(y)$}
\includegraphics[height=4.5cm]{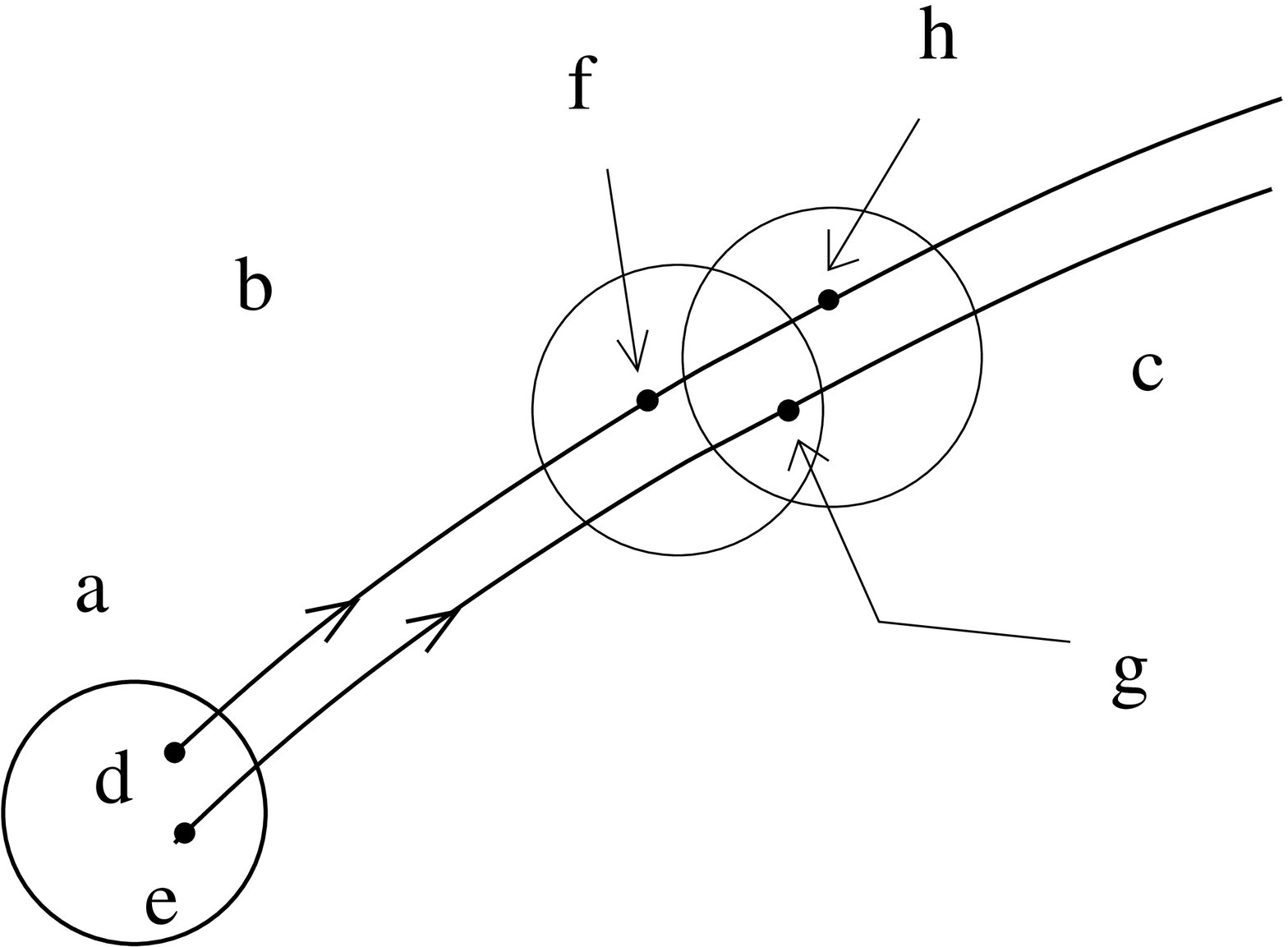}
\caption{} \label{fig1}
\end{center}
\end{figure}

\begin{defi}\label{def-cw}
A flow $X^t$ is \emph{$cw$-expansive} if for any $\epsilon>0$ there is a $\delta>0$ such that if $A\subset M$ is a continuum and $\alpha\in \cH(A)$ satisfies 
$$\diam(\mathcal{X}^t_\alpha(A) )<\delta, \ \ \forall t\in\mathbb{R},$$
then $A\subset X^{(-\epsilon,\epsilon)}(x_\alpha)$.
\end{defi}

\remark Note that in the above definition, as $\alpha(x_\alpha)=id_\RR$ the orbit 
$\cO(x_\alpha) $ guides  the
shearing effect caused by displacement of $A$ by the set of maps in $\cH(A)$.
\vspace{0.2cm}

Our main result is

\begin{maintheorem}\label{Th-$cw$}
If $X^t$ is a $cw$-expansive flow and the topological dimension of $M$ is greater than $1$, then the topological entropy of $X^t$ is positive.
\end{maintheorem}

This paper is organized as follows: in Section \ref{sec-basic} we prove basic properties satisfied by $cw$-expansive flows; in Section $3$ we use Keynes and Sears's  techniques to find good properties in cross sections of $cw$-expansive flows; in Section $4$ we introduce the notion of entropy for special pairs of cross sections and we relate the entropy of this cross sections with the entropy of the flow; in Section $5$ we prove our main theorem.

\section{Basic properties of $cw$-expansive Flows}\label{sec-basic}

Let $M$ be a compact metric space and $X^t$ a $cw$-expansive flow in $M$.   
In this section we prove basic results satisfied by $X^t$.
Given $x \in M$, let $\cO(x)=\{ X^t(x), \, t \in \RR\}$ be the orbit of $x$. A $x$ is a \emph{point fixed} of $X^t$ if $\cO(x)=\{ x\}$.  
Recall that $\cO(x)$ is regular if it is not a unique point. 

\begin{lemma} IF $X^t$ is a $cw$-expansive flow, then each fixed point of $X^t$ is not accumulated by regular orbits of $X^t$. 
\end{lemma} 
\begin{proof} The proof goes by contradiction. Let $X^t(p)=p$ for all $t\in\mathbb{R}$, and suppose $p$ is acumulated by regular orbits. Let $\epsilon=1$ and let $\delta>0$ be the corresponding number satisfying the $cw$-expansive definition. Define
$$h(t)=\left\{\begin{array}{rc} 
t+1, \ \mbox{if} \ t\notin(-2,1) \\
2t, \ \mbox{if} \ t\in [0,1) \\
\frac{t}{2},  \ \mbox{if} \ t\in(-2,0)
\end{array}\right. \,.
$$  

By continuity of the flow respect to initial data, there is $x\in M-\{p\}$ such that if $|t|<3$ then $d(X^t(x),p)<\frac{\delta}{2}$. If $A=X^{[0,1]}(x)$, then $A$ is a continuum. For each $t\in[0,1]$ define  $h_{t}:\mathbb{R}\rightarrow\mathbb{R}$ by
$$ h_{t}=(1-t)Id+th.$$

Now define $\alpha: A \to \Hom(\mathbb{R},0)$ by $\alpha(X^t(x))= h_t$.
Then $\alpha\in \cH(A)$, $X_\alpha^s(A)\subset B_{\frac{\delta}{2}}(x)$ if $s\in (-2,1)$ and this set reduces to a unique point if $s\notin (-2,1)$. Therefore, $\diam(\mathcal{X}_\alpha^s(A))<\delta$ $\forall s\in\mathbb{R}$, contradicting that $X^t$ is $cw$-expansive.

\end{proof}

\bigskip
\noindent Let
$ \quad \Sigma_\mathbb{R}(0)=\{\{x_i\}_{i\in\mathbb{Z}}; x_i\in\mathbb{R} \ \mbox{and} \ x_0=0\}.
$
If $A\subset M$ define

\begin{eqnarray*} \SQ(A)=\{\beta:A\rightarrow\Sigma_\mathbb{R}(0); \exists \,\, x_\beta \ \mbox{such that} \ \beta(x_\beta)_i\rightarrow\infty \ \mbox{when} \ i\rightarrow\infty \ \mbox{and} \\ \beta(x_\beta)_i\rightarrow-\infty \ \mbox{when} \ i\rightarrow-\infty\}.
\end{eqnarray*}
Let $\SQ^*(A)\subset \SQ(A)$ so that if $\beta\in \SQ^*(A)$, then for each  $i\in\mathbb{Z}$ the map 
\begin{eqnarray*} f_i:A\rightarrow\mathbb{R}, \quad \quad
f_i(a)=\beta(a)_i \quad \mbox{is continuous}.
\end{eqnarray*}
If $\beta\in \SQ^*(A)$ and $i\in\mathbb{Z}$ define $\mathcal{X}_\beta^i(A)=\{X^{\beta(x)(i)}(x); \ x\in A\}$.

\bigskip

\begin{thm} Let $X^t$ be a flow without fixed points. The following properties are equivalent:
\begin{enumerate}
\item[(1)] $X^t$ is $cw$-expansive;
\item [(2)] $\forall\,\,\eta>0$ there is $\delta>0$ so that if $A$ is a continuum and there is $\alpha\in \cH(A)$ with $\diam(\mathcal{X}^t_\alpha(A))<\delta$ $\forall t\in\mathbb{R}$, then $A$ is an orbit segment inside $B_\eta(x_\alpha)$;
\item[(3)] $\forall\epsilon>0$ there is $\delta>0$ such that if $A$ is a continuum and exists $\beta\in SQ^*(A)$, with $\beta(x_\beta)_{i+1}-\beta(x_\beta)_{i}\leq \eta$ and $\sup_{a\in A}|\beta(a)_{i+1}-\beta(a)_{i}|\leq \delta$ $\forall i\in\mathbb{Z}$, such that $\diam(\mathcal{X}^i_\beta(A))<\delta$ $\forall i\in\mathbb{Z}$, then $A\subset X^{(-\epsilon,\epsilon)}(x_\beta)$.
\end{enumerate}
\end{thm}
\begin{proof} $(1) \Rightarrow (3):$ Let $\epsilon>0$ be given and $\delta>0$ by $cw$-expansivity property. Choose $\eta>0$ such that 
$$\eta+(2\sup_{z\in M, |u|<\eta}d(z,X^u(z)))<\frac{\delta}{2}. $$
Suppose $A$ is a continuum and that for some $\beta\in \SQ^*(A)$, with $\beta(x_\beta)_{i+1}-\beta(x_\beta)_{i}\leq \eta$ and $\sup_{a\in A}|\beta(a)_{i+1}-\beta(a)_{i}|\leq \delta$  $\forall i\in\mathbb{Z}$, we have $\diam(\mathcal{X}^i_\beta(A))<\eta$, $\forall i\in\mathbb{Z}$.  For each $a\in A$ we define a homeomorphism $\alpha(a)$ by $\alpha(a)(\beta(x_\beta)_i)=\beta(a)_i$, $\forall i\in\mathbb{Z}$ and by linearity on each interval $(\beta(x_\beta)_i,\beta(x_\beta)_{i+1})$. Then $\alpha\in \cH(A)$. 

Therefore, if $t_i=\beta(x_\beta)_i$, for each $t\in[t_i,t_{i+1})$,
\begin{eqnarray*} \diam(\mathcal{X}_\beta^t(A))&=&\sup_{a,b\in A} d(X^{\beta(a)(t)}(a),X^{\beta(b)(t)}(b)) \\
&\leq&\sup_{a,b\in A} (d(X^{\beta(a)(t)}(a),X^t(x_\beta))+d(X^t(x_\beta),X^{\beta(b)(t)}(b))) \\
&\leq&\sup_{a\in A}d(X^{\beta(a)(t)}(a),X^t(x_\beta))+\sup_{b\in A}d(X^t(x_\beta),X^{\beta(b)(t)}(b))) \\
&=&2\sup_{a\in A}d(X^t(x_\beta), X^{\beta(a)(t)}(a))\,.
\end{eqnarray*} 

But, for each $a\in A$ we have that
\begin{eqnarray*} d(X^t(x_\beta), X^{\beta(a)(t)}(a))&\leq&d(X^t(x_\beta),X^{t_i}(x_\beta))+d(X^{t_i}(x_\beta),X^{\beta(a)_i}(a)) \\ &+& d(X^{\beta(a)_i}(a),X^{\beta(a)(t)}(a)) \\
&\leq& \sup_{z\in M, |u|\leq\eta}d(z,X^u(z))+\eta+\sup_{z\in M, |u|\leq\eta}(z,X^u(z)) \\
&<& \frac{\delta}{2}\,.
\end{eqnarray*}
Therefore,
\begin{eqnarray*} \diam(\mathcal{X}_\beta^t(A))< 2\frac{\delta}{2}=\delta,\quad \forall t\, \in\mathbb{R},
\end{eqnarray*}
and since $X^t$ is $cw$-expansive we obtain $A\subset X^{(-\epsilon,\epsilon)}(x_\beta)$.
\vspace{0.2cm}

$(3)\Rightarrow (1)$: Let $\epsilon>0$ be given and $\delta>0$ as in item $(3)$. Suppose $A\subset M$ is a continuum such that for some $\alpha\in \cH(A)$ it holds $\diam(\mathcal{X}^t_\alpha(A))<\delta$ $\forall t \in\mathbb{R}$. By induction, define $\beta\in \SQ^*(A)$ by $\beta(x)_0=0$ $\forall x\in A$. Since $A$ is compact there is $t_1>0$ such that $\alpha(x)(t_1)<\delta$ $\forall x\in A$. Define $\beta(x)_1=\alpha(x)(t_1)$. There is $t_{i+1}>t_i$ such that $|\alpha(x)(t_{i+1})-\alpha(x)(t_i)|<\delta$ $\forall x\in A$, define then $\beta(x)_{i+1}=\alpha(x)(t_{i+1})$. We can define $\beta(x)_i$ for $i<0$ similarly. Therefore $\beta\in \cH(A)$, with $x_\beta=x_\alpha$. By item $(3)$  $A\subset X^{(-\epsilon,\epsilon)}(x_\alpha)$. \vspace{0.2cm}    

$(1)\Rightarrow (2)$: Since $M$ is compact, $\forall \eta>0$ there is $\epsilon>0$ such that $X^{(-\epsilon,\epsilon)}(x)\subset B_\eta(x)$ $\forall x\in M$.\vspace{0.2cm}

$(2)\Rightarrow(1)$:  Since $X^t$  has no fixed points, $\forall\epsilon>0$ there is $\eta>0$ so that $\diam(X^{(-\epsilon,\epsilon)}(x))>\eta$ $\forall x\in M$. 
\end{proof}

Recall that if $M$ and $N$ are metric spaces,
the flows $X^t: M \to M$ and $Y^t:N \to N$ are \emph{conjugate} if there is a homeomorphism  $h:M\to N$ mapping orbits of $X^t$ onto orbits of $Y^t$.

\begin{thm} $cw$-expansiviness is a conjugacy invariant. 
\end{thm}

\begin{proof}  Suppose $X^t$ and $Y^t$ are conjugate with $Y^t$ $cw$-expansive. Let $h:M\rightarrow N$ be a homeomorphism conjugating $X^t$ to $Y^t$. Let $\epsilon_M>0$ be given and $\epsilon_N>0$ such that if $x,y\in N$ satisfies $d(x,y)<\epsilon_N$, then $d(h^{-1}(x),h^{-1}(y))<\epsilon_M$. Let $\delta_N>0$ be the corresponding number given by $cw$-expansivity of $Y^t$ to $\epsilon_N$ and   $\delta_M>0$ such that if $x,y\in M$ with $d(x,y)<\delta_M$, then $d(h(x),h(y))<\delta_N$. Therefore, if $A_M\subset M$ is a continuum and $\alpha_M\in \cH(A_M)$ is such that $\diam(\mathcal{X}^t_{\alpha_M}(A_M))<\delta_M$ $\forall t\in\mathbb{R}$ we get that $A_N=h(A_M)\subset N$ is a continuum. 
For every $x\in A_M$ and $t\in\mathbb{R}$ let $\alpha_N(h(x))(t)$ be the real number such that if
$$h(X^{\alpha_M(x)(t)}(x))=Y^{\alpha_N(h(x))(t)}(h(x))$$
then $\alpha_N\in \cH(A_N)$. Furthermore, for every $t\in\mathbb{R}$ we have:
\begin{eqnarray*} \diam(\mathcal{Y}^t_{\alpha_N}(A_N))&=&\max_{x,y\in A_N}d(Y^{\alpha_N(x)(t)}(x),Y^{\alpha_N(y)(t)}(y)) \\
&=& \max_{x,y\in A_M}d(h(X^{\alpha_M(x)(t))}(x)),h(X^{\alpha_M(y)(t))}(y))) \\
&=& \diam(h(\mathcal{X}^t_{\alpha_M}(A_M))) <\delta_N,\\
\end{eqnarray*}
because  $\diam(\mathcal{X}^t_{\alpha_M}(A_M))<\delta_M$. 
Since $Y^t$ is $cw$-expansive, $A_N$ is an $Y^t$-orbit segment inside $B_{\epsilon_N}(x_{\alpha_N})$. But $A_M=h^{-1}(A_N)$, $x_{\alpha_N}=h(x_{\alpha_M})$ and the choice of $\epsilon_N$ imply that $A_M$ is a $X^t$-orbit segment inside $B_{\epsilon_M}(x_{\alpha_M})$, proving that $X^t$ is $cw$-expansive.   

\end{proof}
\section{Suspensions}

Let $(M,d)$ be a compact metric space and
 $f:M\rightarrow M$ a homeomorphism. Let $k:M\rightarrow\mathbb{R}^+$ be a continuous function. 

\begin{defi}The \emph{suspension} of $f$ under $k$ is the flow $X^t$ on the space 
$$M_k=\bigcup_{0\leq t\leq k(y)}(y,t)/(y,k(y))(k(y),0)$$
defined for small nonnegative time by $X^t(y,s)=(t+s)$, $0\leq t+s< k(y)$. 
\end{defi}
Each suspension of $f$ is conjugate to the suspension of $f$ under $1$, the constant function with value $1$. For this reason we shall concentrate on suspensions under the function $1$. 

Next, following \cite{BW}, we define a metric on $M_1$. Suppose that the diameter of $M$ under $d$ is less than $1$. 

Consider the subset $M\times\{t\}$ of $M\times[0,1]$ and let $d_t$ denote the metric defined by $d_t((y,t),(z,t))=(1-t)d(y,t)+td(f(y),f(z))$, $y,z\in M$. 
Given $x_1,x_2\in M_1$, consider all finite chains $x_1=w_0,w_1,...,w_n=x_2$ between $x_1$ and $x_2$ where, for each $i$, either $w_i$ and $w_{i+1}$ belong to $M\times{t}$ for some $t$ (in wich case we call $[w_i,w_{i+1}]$ a horizontal segment) or $w_i$ and $w_{i+1}$ are on the same orbit (and then we call $[w_i,w_{i+1}]$ a vertical segment). 

Define the length of a chain as the sum of the lengths of its segments, where the length of a horizontal segment $[w_i,w_{i+1}]$ is measured in the metric $d_t$ if $w_i$ and $w_{i+1}$ belongs to $M\times\{t\}$, and the length of a vertical segment $[w_i,w_{i+1}]$ is the shortest distance between $w_i$ and $w_{i+1}$ along the orbit (ignoring the direction of the orbit) using the usual metric on $\mathbb{R}$. 

If $w_i\neq w_{i+1}$ and $w_i$ and $w_{i+1}$ are on the same orbit and on the same set $M\times\{t\}$ then the length of the segment $[w_i,w_{i+1}]$ is taken as $d_t(w_i,w_{i+1})$, since this is always less than $1$. 

Then define $d(x_1,x_2)$ to be the infimum of the lengths of all chains between $x_1$ and $x_2$. It is easy to see that $d$ is a metric on $M_1$.
This metric $d$ gives the topology on $M_1$.

\begin{thm} \label{th-sus} Let $\phi:Y\rightarrow Y$ be a homeomorphism and $f:M\rightarrow\mathbb{R}^+$ a continuous map. The suspension of $\phi$ under $f$ is $cw$-expansive if and only if $\phi$ is $cw$-expansive.
\end{thm}

\begin{proof}We need only show the result when $f\equiv 1$.
Suppose that $X^t$ is $cw$-expansive. Let $\frac{1}{2}>\epsilon>0$ be given and $\delta>0$ be the corresponding constant determined by  $cw$-expansivity. If $A\subset Y$ is a continuum with $\diam(\phi^n(A))<\delta$ $\forall n\in\mathbb{Z}$, denoting  $A_M=A\times\{0\}$, and $x_1=(x,0)$, then for all $t\in\mathbb{R}$ we have:
\begin{eqnarray*} \diam(X^t(A_M))&=&\max_{x_1,y_1\in A_M}d(X^t(x_1),X^t(y_1)) \\
&\leq& \max_{x,y\in A}\rho_{t-[t]}((\phi^{[t]}(x),t-{[t]}),(\phi^{[t]}(y),t-{[t]})) \\
&=& \max_{x,y\in A}((1-t+[t])\rho(\phi^{[t]}(x),\phi^{[t]}(y))+(t-[t])\rho(\phi^{[t]+1}(x),\phi^{[t]+1}(y))).
\end{eqnarray*} 
Where $[t]$ is the greatest integer less than $t$. But since  for all $ x,y\in A$ it holds
$$\phi^{[t]}(x),\phi^{[t]}(y)\in\phi^{[t]}(A) \ \mbox{and} \ \phi^{[t]+1}(x),\phi^{[t]}(y)\in\phi^{[t]+1}(A)\, ,$$
we get
\begin{eqnarray*} \diam(X^t(A_M))&\leq&(1-t+[t])\delta+(t-[t])\delta=\delta. 
\end{eqnarray*}
Since $X^t$ is $cw$-expansive, there is $(x,0)$ such that $A_M\subset X^{(-\epsilon,\epsilon)}(x)$. Moreover since $0<\epsilon<\frac{1}{2}$ and $A_M\subset M\times\{0\}$ we obtain $A_M=\{(x,0)\}$, and so $A=\{x\}$. Therefore, $\phi$ is $cw$-expansive.

Next suppose that the suspension $\phi$ is  $cw$-expansive. Consider in $M$ the metric given by
$$\rho'(x,y)=\min\{\rho(x,y),\rho(\phi(x),\phi(y))\}$$
and let $\delta>0$ be the  $cw$-expansivity constant to $\rho'$. Let $\epsilon>0$ and $\delta'=\min\{\delta,\epsilon,\frac{1}{4}\}$. 
Let $A\subset M_f$ be a continuum and $\alpha\in \cH(A)$ so that $\diam\mathcal{X}^t_\alpha(A)<\delta'$ $\forall t\in\mathbb{R}$. 
We consider two cases: (1)  $x_\alpha$ can be represented as $(y_1,\frac{1}{2})$ and (2) $x_\alpha$ can be represented as $(\frac{1}{2},y_1)$.

In the first case, define $A_M=\{a\in M;(a,s)\in A \ \mbox{com} \ s\in(0,1) \}$. Then
\begin{eqnarray*}\diam(A_M)&=&\max_{y,z\in A_M}\rho'(y,z) \\
&\leq& \max_{(y,s),(z,r)\in A}d((y,s),(z,r)) \\
&=&\diam(A)<\delta'<\delta.
\end{eqnarray*}  
By definition of suspension we have that $X^{1}(x_\alpha)$ has representation $(\phi(y_1),\frac{1}{2})$, and since $\diam \mathcal{X}^1_\alpha(A)<\delta<\frac{1}{4}$ we get that $X^{\alpha(y)(1)}(y)$ has representation $(\phi(y),s)$ with $s\in(0,1)$ $\forall y\in A$. So,
\begin{eqnarray*}\diam(\phi(A_M))&=&\max_{y,z\in A_M}\rho'(\phi(y),\phi(z)) \\
&\leq& \max_{(\phi(y),s),(\phi(z),r)\in \mathcal{X}_\alpha^1(A)}d((\phi(y),s),(\phi(z),r)) \\
&=&\diam(\mathcal{X}^1_\alpha(A))<\delta'<\delta.
\end{eqnarray*}
Similarly, $\diam(\phi^n(A_M))<\delta$ $\forall n\in\mathbb{Z}$. Since 
$\phi$ is $cw$-expanse  we get $A_M=\{y_1\}$, and so, every point in $A$ has the form $(y_1,t)$ with $t\in(0,1)$. Since $(y_1,\frac{1}{2})=x_\alpha\in A$ and $\diam(A)<\delta'<\epsilon$, we conclude that $|t-\frac{1}{2}|<\epsilon$, ie, $A\subset X^{-\epsilon,\epsilon}(x_\alpha)$, finishing the proof in the first case.
\vspace{0.2cm}

When $x_\alpha$ does not have a representation as $(y_1,\frac{1}{2})$ then  $X^r(x_\alpha)$ has representation as $(y_1,\frac{1}{2})$ for some $|r|<\frac{1}{2}$. Define $\widetilde{A}=\mathcal{X}^r_\alpha(A)$ and for each $x\in A$ and $t\in\mathbb{R}$ set
$$\widetilde{\alpha}(X^r(x))(t)=\alpha(x)(t+r)-\alpha(x)(r).$$
We have that $\widetilde{A}$ is a continuum and $\widetilde{\alpha}\in \cH(\widetilde{A})$. Moreover,  for every $t\in\mathbb{R}$, it holds
\begin{eqnarray*}\diam(\mathcal{X}^t_{\widetilde{\alpha}}(\widetilde{A}))&=&\max_{x,y\in \widetilde{A}}d(X^{\widetilde{\alpha}(x)(t)}(x),X^{\widetilde{\alpha}(y)(t)}(y)) \\
&=&\max_{x,y\in A}d(X^{\alpha(x)(t+r)-\alpha(x)(r)}(X^{\alpha(x)(r)}(x)),X^{\alpha(y)(t+r)-\alpha(y)(r)}(X^{\alpha(y)(r)}(y)) \\
&=&\max_{x,y\in A}d(X^{\alpha(x)(t+r)}(x),X^{\alpha(y)(t+r)}(y)) \\
&=&\diam(\mathcal{X}^{t+r}_\alpha(A))<\delta.
\end{eqnarray*}
By the first case, we obtain 
$$\widetilde{A}\subset X^{(-\epsilon,\epsilon)}(x_{\widetilde{\alpha}})=X^{(-\epsilon,\epsilon)}(X^r(x_\alpha)),$$
and therefore,
$$A=X^{-r}(\widetilde{A})\subset X^{(-\epsilon,\epsilon)}(x_\alpha).$$
All together completes the proof of Theorem \ref{th-sus}.

\end{proof}

\section{Pairs of cross-sections families}\label{sec-pairs}

In this section we consider a $cw$-expansive flow $X^t$ on $M$ and 
define stable and unstable sets for subsets $A \subset M$. 
The goal is to extend to this context \cite[Theorem 2.7]{KS}, establishing that
locally these sets defined for points of $M$ intersect at a unique point. 
To do so we use 
the notation introduced by Keynes and Sears in \cite{KS} and start recalling the definition
of $\delta$-adapted family of cross-sections. \vspace{0.1cm}

A set $S\subset M$ is a \emph{cross-section} of time $\epsilon>0$ if it is closed and for each $x\in S$ we have $S\cap X^{(-\epsilon,\epsilon)}(x)=\{x\}$. The \emph{interior} of $S$ is the set $S^*=int(X^{(-\epsilon,\epsilon)}(S))\cap S$.  \vspace{0.1cm}

The proof of the next lemma can be found in \cite[Lemma 2.4]{KS} as well in \cite[Lemma 7]{BW}.

\begin{lemma}\label{lst} There is $\epsilon>0$ such that for each $\delta>0$ we can find a pair $(\mathcal{S},\mathcal{T})$ of finite families $\mathcal{S}=\{S_1,...,S_n\}$ and $\mathcal{T}=\{T_1,...,T_n\}$ of local cross-sections of time $\epsilon>0$ and diameter at most $\delta$ with $T_i\subset S_i^*$ ($i\in\{1,...,k\}$) such that
$$M=\bigcup_{i=1}^{k} X^{[0,\epsilon]}(T_i)=\bigcup_{i=1}^{k} X^{[-\epsilon,0]}(T_i)=\bigcup_{i=1}^{k} X^{[0,\epsilon]}(S_i)=\bigcup_{i=1}^{k} X^{[-\epsilon,0]}(S_i).$$   
\end{lemma}
  
\begin{defi}\label{adequada}
A pair of families of cross-sections $(\mathcal{S},\mathcal{T})$  as in the previous theorem is called {\emph{$\delta$-adequated}}.
\end{defi}
Given a  pair of   $\delta$-adequated cross sections $(\mathcal{S},\mathcal{T})$ let

\begin{equation}\label{e.theta} 
\theta=\sup\{\delta>0; \forall x\in\bigcup_{i=1}^{k}S_i \ \mbox{it holds} \ X^{(0,\delta)}(x)\cap\bigcup_{i=1}^{k}S_i=\emptyset\}\,.
\end{equation}

Let $\rho>0$ satisfying $5\rho<\epsilon$ and $2\rho<\theta$. And for each $S_i$ consider $D_\rho^i=X^{(-\rho,\rho)}(S_i)$ and define the projection 
\begin{equation}\label{e.projecao} P_\rho^i:D_\rho^i\rightarrow S_i
\end{equation}
by $P_\rho^i(x)=X^t(x)$, were $X^t(x)\in S_i$ for $|t|<\rho$. Let $\frac{1}{2}\theta>\epsilon_0>0$ be such that if $x,y\in S_i$, $d(x,y)<\epsilon_0$ and $t$ is a real number with $|t|<3\delta$ and $X^t(x)\in T_j$, then $X^t(y)\in D^j_\rho$.

Let  $\phi:\bigcup_{i=1}^k T_i\to \bigcup_{i=1}^k T_i$ be  the
{\em{ first return map}} defined as $\phi(x)=X^t(x)$ where $t>0$ is the smallest positive number such that $X^t(x)\in \bigcup_{i=1}^k T_i$. Note that $t\in[\theta,\epsilon]$.

\begin{figure}[htb]
\begin{center}
\psfrag{a}{$x$}
\psfrag{b}{$y$}
\psfrag{c}{$A$}
\psfrag{d}{$\phi(A,y)$}
\psfrag{e}{$\phi^2(A,y)=\phi(A,x)$}
\includegraphics[height=4.5cm]{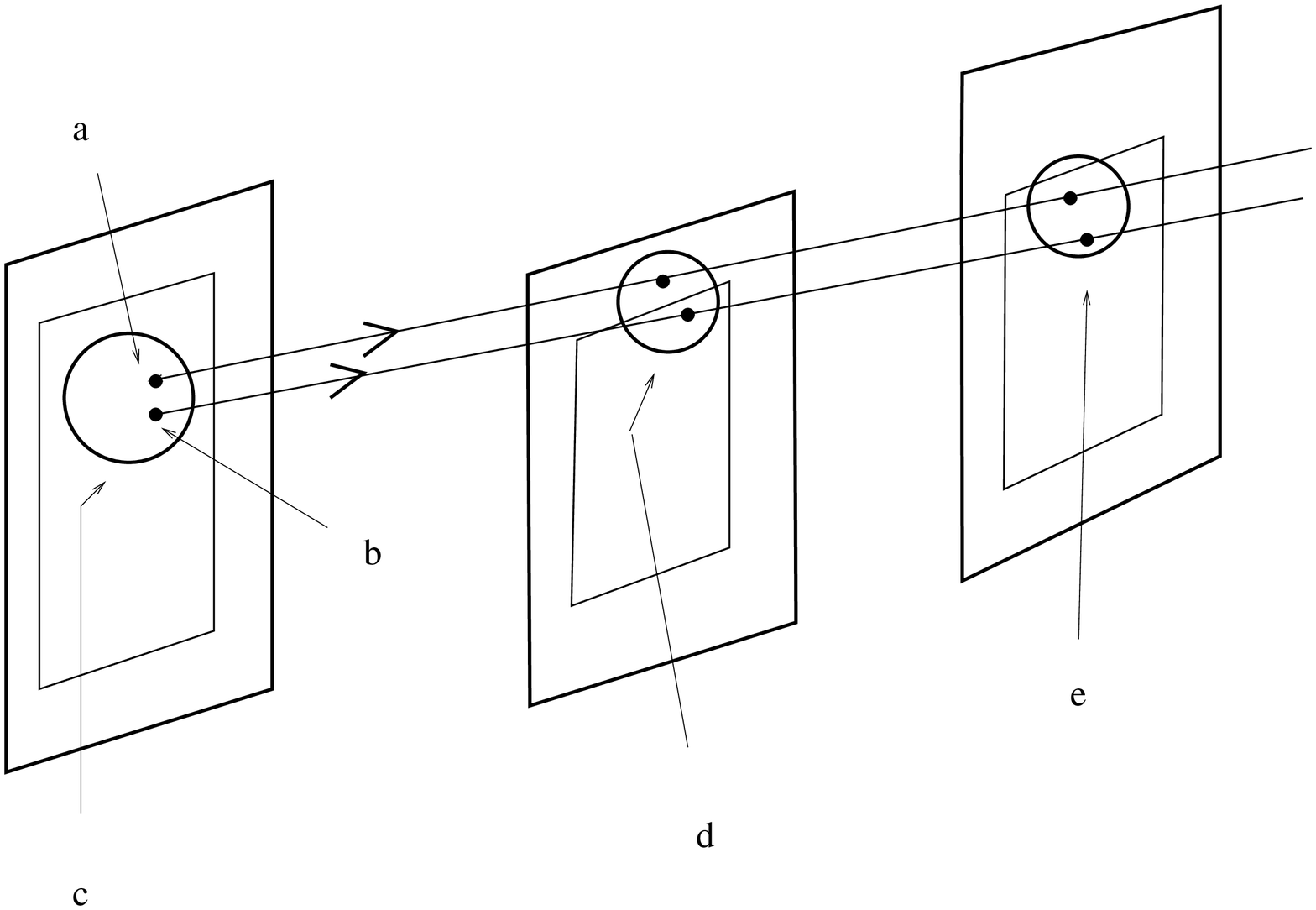}
\caption{} \label{fig2}
\end{center}
\end{figure}

If $x\in T_i$ and $y\in S_i$ with $d(x,y)<\epsilon_0$ let $\{y^x_0,...,y^x_n\}\subset \cO_X(y)$ such that $y^x_0=y$ and $y^x_j=P_\rho^l(X^t(y^x_{j-1}))$, where $t>0$ is the smallest positive time such that $\phi^j(x)=X^t(\phi^{j-1}(x))$, and $l$ is such that $\phi^j(x)\in T_l$. We can continue this construction while $d(\phi^j(x),y_j)<\epsilon_0$. Similarly to $j<0$.

The stable and unstable sets of points is defined in the following way.
If $x\in T_i$ and $\eta<\epsilon_0$, the \emph{$\eta$-estable set} of $x$ is
is defined as
$$ W^s_\eta(x)=\{y\in S_i; d(\phi^i(x),y_i)<\eta \ \forall i\geq 0\}$$       
and the \emph{$\eta$-instable set} of $x$ is defined as
$$ W^u_\eta(x)=\{y\in S_i; d(\phi^i(x),y_i)<\eta \ \forall i\leq 0\}\,.$$

\bigskip
Theorem 2.7 in \cite{KS} establishes that a 
 flow $X^t$ is expansive if, and only if, given a pair $(\mathcal{S},\mathcal{T})$ $\delta$-adequate, there is an $\eta>0$ such that  $W^s_\eta(x)\cap W^u_\eta(x)=\{x\}$ for every $ x\in\bigcup_{i=1}^kT_i$,.  

\bigskip

To extend this result in the context of $cw$-expansive flows, we first introduce the
notion of stable and unstable sets of pair of $\delta$-adapted cross sections. For this, 
let $A\subset S_j$ and $x\in A\cap T_j$. If $\diam A<\frac{1}{2}\epsilon_0$ then $X^t(A)\subset D_\rho^i$, where $t$ and $i$ are such that $X^t(x)=\phi(x)\in T_i$. Define 
$$\phi(A,x)=\{P_\rho^i(X^t(y)); y\in A\}\subset S_i.$$

\begin{figure}[htb]
\begin{center}
\psfrag{a}{$A$}
\psfrag{b}{$x$}
\psfrag{c}{$A_0$}
\psfrag{d}{$\phi^N(A,x)$}
\psfrag{e}{$A_1$}
\psfrag{w}{$y$}
\psfrag{v}{$\phi^N(x)$}
\psfrag{k}{$A_{1,1}$}
\psfrag{l}{$\phi^N(y)$}
\psfrag{f}{$A_{0,0}$}
\psfrag{g}{$\phi^N(A_0)$}
\psfrag{h}{$A_{0,1}$}
\psfrag{i}{$A_{1,0}$}
\psfrag{j}{$\phi^N(A_1,y)$}
\includegraphics[height=4.5cm]{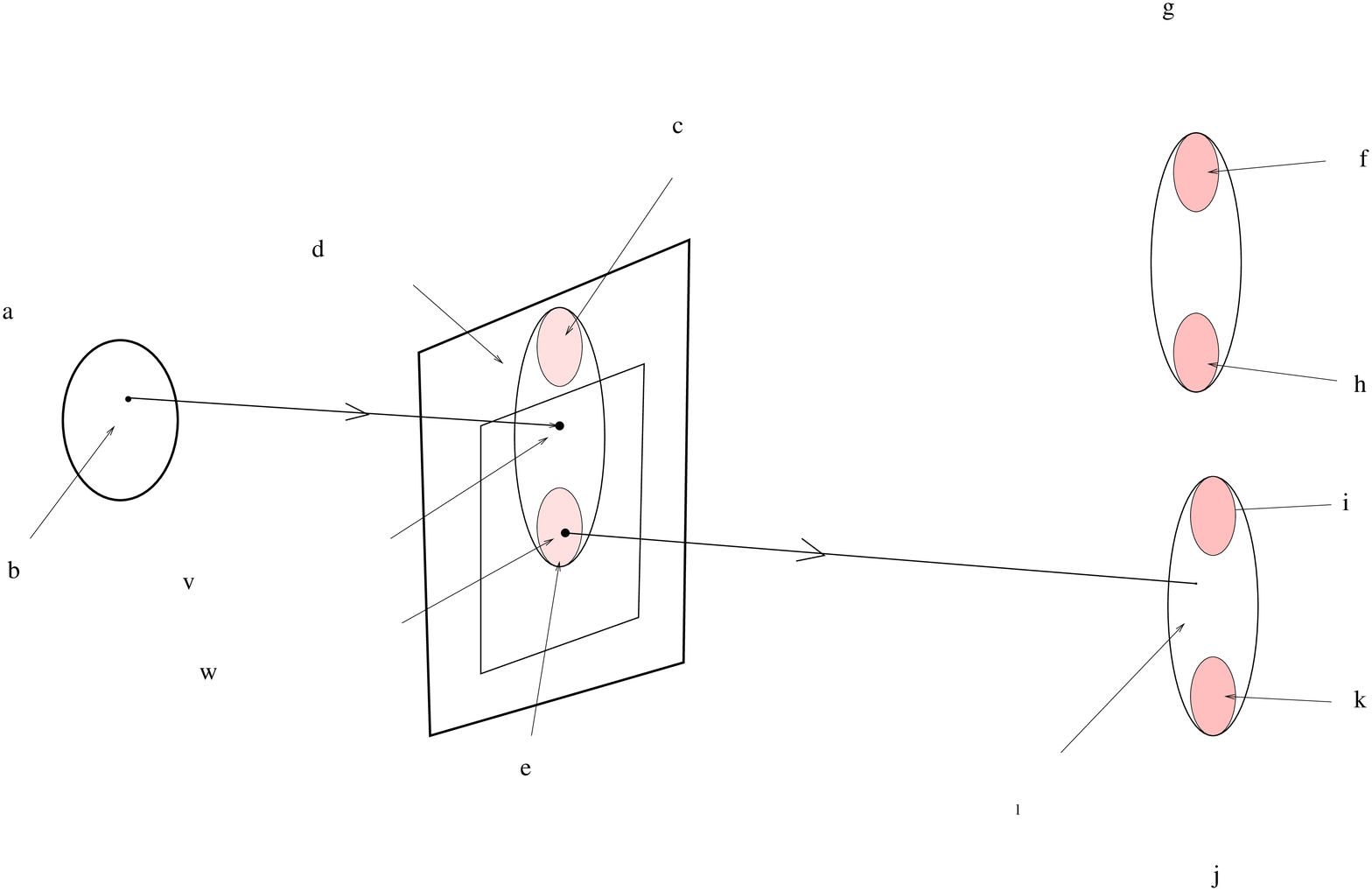}
\caption{} \label{fig3}
\end{center}
\end{figure}

Fix $n\in\mathbb{N}$. If $\phi^n(A,x)$ is defined with $\diam(\phi^n(A))<\epsilon_0$ and $\phi^n(x)\in T^i$, then $X^t(\phi^n(A))\in D_\rho^l$, where $t$ and $l$ are such that $X^t(\phi^n(x))=\phi^{n+1}(x)\in T^l$. Define
$$\phi^{n+1}(A,x)=P_\rho^l(X^t(\phi^n(A)).$$

\begin{defi}\label{variedade-estavel-par}
For $\epsilon_0>\eta>0$ define the local \emph{stable and unstable sets} with respect to  $(\mathcal{T},\mathcal{S})$ by
$$W^s_\eta(\mathcal{S},\mathcal{T})=\{A\in C(M); \mbox{there is} \ x\in A \ \mbox{such that} \ \diam(\phi^n(A,x))<\eta \ \forall n\geq 0\}\, ,$$
$$W^u_\eta(\mathcal{S},\mathcal{T})=\{A\in C(M); \mbox{there is} \ x\in A \ \mbox{such that} \ \diam(\phi^n(A,x))<\eta \ \forall n\leq 0\} \, .$$
\end{defi}
Here, recall that $C(M)$ is the set of compact and connected sets of $M$.

\begin{thm} A flow $X^t$ is $cw$-expansive if, and only if, given a pair $(\mathcal{S},\mathcal{T})$ $\delta$-adequate 
there is an $\eta>0$ such that  
$$W_\eta^s(\mathcal{S},\mathcal{T})\cap W_\eta^u(\mathcal{S},\mathcal{T})=\{A\subset M; \# A=1\}.$$  

\end{thm}
\begin{proof}  Suppose that $X^t$ is $cw$-expansive. Let $a\in(0,\epsilon)$ be given and $a_1>0$ be the $cw$-expansive constant. Let $\eta\in(0,\epsilon_0)$ be such that if $p\in T_i$ and $q\in S_i$, $i\in\{1,...,k\}$, with $d(p,q)<\eta$, then $d(X^t(p),X^s(q))<\frac{a_1}{2}$ where, if $X^{t_1}(p)=\phi(p)$ and $X^{s_1}(q)=q_1$, then $t\in[0,t_1]$ and $s\in[0,s_1]$ and $|s-t|\leq |s_1-t_1|$. Suppose $A\subset W^s_\eta(\mathcal{S},\mathcal{T})\cap W^u_\eta(\mathcal{S},\mathcal{T})$, and let $x_s\in A$ be given by the  definition of $W^s_\eta(\mathcal{S},\mathcal{T})$ and $x_u$ be given by the definition of $W^u_\eta(\mathcal{S},\mathcal{T})$ and fix $x\in A$. Define $t_0=0$, and $t_1>0$ as the time that $X^{t_1}(x)\in\phi(A,x_s)$ and $t_1\in (r_1-\rho,r_1+\rho)$ with $r_1$ the smallest positive time such that $\phi(x_s)=X^{r_1}(x_s)$. If $t_n$ and $r_n$ are well defined, we define $t_{n+1}$ as the real positive number such that $X^{t_{n+1}}(x)\in\phi_{n+1}(A,x_s)$ and $t_{n+1}\in(\sum_{i=1}^nr_i-\rho,\sum_{i=1}^nr_i+\rho)$ with $r_{n+1}$ the smallest positive real number such that $\phi^{n+1}(x^s)=X^{r_{n+1}}(\phi^{n}(x^s))$. Similarly, we can define $t_n$ for $n<0$. For each $y\in A$ we can, with the same process, find a real number sequence $s_n^y$. Define then $\alpha(y)$ by $\alpha(y)(t_n)=s_n$ and linearly on each interval $(t_n,t_{n+1})$.

Then for each $t\in\mathbb{R}$ and $y\in A$ we have $t=t_n+\sigma=s_n^y+\sigma_y$, where $\sigma\in [0,t_{n+1}-t_n]$, $\sigma_y\in [0,s_{n+1}^y-s_n^y]$ and 
$$|\sigma-\sigma_y|\leq|(t_{n+1}-t_n)-(s_{n+1}^y-s_n^y)|,$$
for any $n\in\mathbb{Z}$. Therefore,
\begin{eqnarray*} \diam(\mathcal{X}^t_\alpha(A))&=&\max_{y,z\in A}d(X^{\alpha(y)(t)}(y),X^{\alpha(z)(t)}(z)) \\
&\leq&\max_{y,z\in A}(d(X^{\alpha(y)(t)}(y),X^t(x))+d(X^t(x),X^{\alpha(y)(t)}(y))) \\
&\leq&2\max_{y\in A}d(X^{\alpha(y)(t)}(y),X^t(x)) \\
&<&2\frac{a_1}{2}=a_1.
\end{eqnarray*}
By the choice of $a_1$ we have that $A\subset X^{(-a,a)}(x)$, and since $a<\rho$ and $A$ is inside a cross section we conclude that  $A=\{x\}$.

To prove the reverse implication, suppose  that given a pair $(\mathcal{T},\mathcal{S})$ and $\rho>0$, there is $\eta>0$ such that 
$$W_\eta^s(\cS, \cT)\cap W_\eta^u(\cS,\cT)=\{A\subset M; \# A=1\}.$$

Given $\epsilon > 0$, let $0<a_1< \min\{\frac{\eta}{2}, \epsilon\}$.
We shall prove that
there is $ \sigma > 0$ such that if  $A\subset M$ is a continuum and $\alpha\in \cH(A)$ with $A\not\subset X^{(-a_1,a_1)}(x_\alpha)$ then there is $t\in \RR$ such that $\diam(\cX_\alpha^t(A)) > \sigma. $ 
For this we split the proof into cases.

\textbf{Case 1:}  $A\subset S\in\mathcal{S}$, 
with the guide point $x_\alpha\in T$. For each $y\in A$ let $t_0=s^y_0=0$ and if $i>0$ define $t_{i+1}=t_i+t$ and $s_{i+1}^y=s_{i}^y+s$ where $t$ and $s$ are the smallest positive times such that $X^{t_{i+1}}(\phi^i(x_\alpha))=\phi^{i+1}(x_\alpha)$ and $X^{s_{i+1}}(y_{i}^{x_\alpha})=y_{i+1}^{x_\alpha}$. Similarly for $i<0$. 

Choose $\delta_0\in(0,\epsilon-\delta-\rho)$ and positive numbers, $a_2<a_1,$ and $a_3,a_4>0$ such that if $u\in T_i$ and $v\in S_i$ then it holds
\begin{enumerate}
\item $d(u,v)<a_1$ implies $d(u,X^t(v))>a_1$ for all $|t|\in[\delta_0,\epsilon]$;
\item $d(u,v)<a_2$ implies $d(\phi(u),v_1)<a_1$;
\item $d(u,v)\geq a_2$ implies $d(u,X^t(v))>a_3$ for all $|t|<\delta$;
\item If $x,y\in M$ and $d(x,y)<a_4$ then $d(X^t(x),X^t(y))<a_1$ for all $|t|<\delta$.
\end{enumerate}   
Let $a'=\min\{a_2,a_3,a_4\}$.  
In this case, we will prove that  $a'$ is a $cw$-expansive constant to $X^ t$. 
For this we proceed as follows.
\begin{description}
\item[(a)] Suppose that for each $i\in\mathbb{Z}$ we have  
$$\sup_{y\in A} |\alpha(y)(t_i)-s_i|<\delta.$$
By hypothesis, there is $n\in\mathbb{Z}$ such that $\diam \phi^n(A)>\eta$. Take then $y\in A$ such that $d(\phi^n(x_\alpha),y_n)>\frac{\eta}{2}$. Therefore,
$$d(X^{t_n}(x_\alpha),X^{s_n}(y))>\frac{\eta}{2}>a_1>a_2\,.$$
And by $(3)$ we get
$$\diam \mathcal{X}^{t_n}_\alpha(A)>d(X^{t_n}(x_\alpha),X^{\alpha(y)(t_n)}(y)>a_3\geq a';$$ 
\item[(b)]Suppose that $j\in\mathbb{Z}$ is the integer  with smallest modulus such that   
$$\sup_{y\in A} |\alpha(y)(t_j)-s_j|\geq\delta.$$
We can assume $j>0$. Let $y\in A$ be such that $|\alpha(y)(t_j)-s_j|\geq\delta$.
\item[(b.1)] Suppose there is $i\in [0,j)$ such that
$$ d(\phi^i(x),y_i)>a_2.$$
Since $X^{t_n}(x_\alpha)$ and $X^{\alpha(y)(t_n)}(y)$ belong to $\mathcal{X}^{t_n}_\alpha(A)$,
reasoning as in (a),  we get
$$\diam \mathcal{X}^{t_n}_\alpha(A)\geq d(X^{t_n}(x_\alpha),X^{\alpha(y)(t_n)}(y)>a_3\geq a'.$$
\item[(b.2)] Suppose that for all $i\in [0,j)$ we have 
$$ d(\phi^i(x),y_i)\leq a_2<a_1.$$
\item[(b.2.1)] Suppose $t=s_j^y-\alpha(y)(t_j)\geq\delta$. If $\alpha(t_j)\geq s_{j-1}-\delta_0$, then
$$\delta_0\leq t\leq s_j-s_{j-1}+\delta_0<\delta+\rho+\delta_0<\epsilon$$
and by (1),
\begin{eqnarray*} \diam\mathcal{X}^{t_j}_\alpha(A)&\geq& d(X^{t_j}(x_\alpha),X^{\alpha(y)(t_j)}(y)) \\
&=& d(X^{t_j}(x_\alpha),X^{(s_j-t)}(y))>a_1.
\end{eqnarray*}
 If $\alpha(t_j)<s_{j-1}^y-\delta$ we can find $t'\in (t_{j-1},t_j)$ such that $\alpha(y)(t')=s_{j-1}-\delta_0$. Then, for  $\zeta=t'-t_{j-1}$, 
 $(1)$ implies that
$$d(X^{t_{j-1}}(x_\alpha),X^{s_{j-1}^y-\delta_0-\zeta}(y)>a_1.$$  
Therefore, by $(4),$ 
\begin{eqnarray*} \diam(X^{t'}(A))&\geq& d(X^{t'}(x_\alpha),X^{\alpha(y)(t')}(y)) \\
&=& d(X^{t_{j-1}+\zeta}(x_\alpha),X^{s_{j-1}^y-\delta_0}(y)))>a_4.
\end{eqnarray*}
\item[(b.2.2)] Suppose $t=\alpha(y)(t_j)-s^y_j\geq\delta_0$. Since
$$s_{j-1}^y+\delta\in[\alpha(y)(t_{j-1}),s_j^y+\delta_0),$$
there is $t'\in(t_{j-1},t_j]$ with $\alpha(y)(t')=s_j^y+\delta_0$. 
Let $\zeta=t_j-t'$. As $d(X^{t_j}(x_\alpha),X^{s_j}(y))<a_1$, $(4)$ implies that
$$d(X^{t_j}(x_\alpha),X^{\zeta+\delta_0}(X^{s_j^y}(y)))>a_1.$$
Therefore, 
\begin{eqnarray*} \diam \mathcal{X}^{t'}_\alpha(A)&\geq& d(X^{t'}(x_\alpha),X^{\alpha(y)(t')}(y)) \\
&=&d(X^{t_j+\zeta}(x_\alpha),X^{s_j+\delta_0}(y))>a_4 > a' . 
\end{eqnarray*}
\end{description}

\textbf{Case 2:} Now suppose that $A$ is not inside  a cross-section of $\mathcal{S}$. 
Since $X^t$ is continuous and $\bigcup_{i=1}^kT_i$ is compact, we can
take $\delta_1>0$ and $a_5>0$ such that if $d(x,y)<a_5$ then (1) and (2) below holds.
\begin{enumerate}
\item $d(X^t(x),X^{s}(y))<\frac{a'}{2}$, where $t>0$ is the smallest positive integer such that $X^t(x)\in \bigcup_{i=1}^kT_i$, $X^s(y)=D^i_\rho(X^t(y))$ and $|t-s|<\frac{\delta_1}{16}$;
\item $d(X^w(x),X^v(y))<\frac{a'}{2}$, for all $|w|,|v|\leq\delta_1+\delta$ and $|v-w|\leq \delta_1$. 
\end{enumerate} 
Let $a_6>0$ be such that if $d(x,y)\leq a_6$ then $d(X^{t'}(x),X^{t'+t}(y))>a_6$ for $|t|\in(\frac{\delta_1}{16},\epsilon)$ and $|t'|<\delta$. 

Take $0<a_7<a_5$ such that the connected component of $\cO(x)\cap B_{a_7}$
that contains $X^t(x)$ 
 is contained in $X^{(-a_5,a_5)}(x)$. 
 
Now suppose that $A$ is not contained in $X^{(-a_5,a_5)}(x_\alpha)$, and let $t$ be the smallest positive time such that $X^t(x_\alpha)\in\bigcup_{i=1}^kT_i$.

\begin{description}
\item[(a)] Suppose that $\forall y\in A$, $|\alpha(y)(t)-s|<\frac{\delta_1}{8}$. Since $A$ is compact there is $0 < \delta' <  \frac{\delta_1}{4}$ such that
$$\sup_{y\in A}|\alpha(y)(t+t')-s|<\frac{\delta_1}{4}$$ 
for all $|t'|\leq\delta'$. Define for each $y\in A$ a homeomorphism $\beta(y)(t)=\alpha(y)(t'+t)-s$ for $|t'|\geq\delta'$, $\beta(y)(0)=0$ and  linearly on $(0,\delta')$. Thus, for all $y\in A$, and  all $|t'|<\delta'$ we get
$$\beta(y)(t')-t'|\leq \frac{\delta_1}{4}+\delta'<\frac{\delta_1}{2}\,.$$    
 By $(1)$ and $(2)$, for every $|t'|<\delta'$ we have
$$d(X^{t+t'}(x_\alpha),X^{s+\beta(y)(t')}(y))<\frac{a'}{2}.$$
\item[(b)] Suppose that there exists $y_0\in A$ such that $|\alpha(y_0)(t)-s|\geq\frac{\delta_1}{8}$. Then $|\alpha(y_0)(t)-t|\geq\frac{\delta_1}{16}$ and there is $t'\leq t$ such that $|\alpha(y_0)(t')-t'|=\frac{\delta_1}{16}$. Then,
\begin{eqnarray*} \diam (\mathcal{X}^{t'}_\alpha(A))&\geq& d(X^{t'}(x_\alpha),X^{\alpha(y_0)(t')}(y_0)) \\
&=& d(X^{t'}(x_\alpha),X^{t'\pm\frac{\delta_1}{16}}(y_0)),
\end{eqnarray*}  
If $d(X^{t'}(x_\alpha),X^{t'\pm\frac{\delta_1}{16}}(y_0))> a_6$, then we are done. 
Otherwise, by the choice of $a_6$, we get
 $d(x_\alpha,y_0)>a_6$, which implies that
  $\diam(A)>a_6$. 

Therefore, $\sigma = \min(a',a_6,a_7)$ is a $cw$-expansive constant to $X^t$ 
relative to $\epsilon$.
\end{description}

\end{proof}

Roughly speaking, the next result shows that the local stable and unstable sets of pair
$(\cS, \cT)$ of $\delta$-adequate families of cross-sections for $cw$-expansive flows behave 
as local stable and unstable sets of points when the flow is expansive, for
all points in a continuum $A \in W^s_\eta(\cS, \cT)$.
\begin{thm}\label{var-estavel} Let $X^t$ be a $cw$-expansive flow. Then
\begin{itemize}
\item[(a)] if $A\in W^s_\eta(\mathcal{S},\mathcal{T})$ then $\forall x\in A$, 
$\diam(\phi^n(A,x))\rightarrow 0$ when $n\rightarrow+\infty$;
\item[(b)] if $A\in W^u_\eta(\mathcal{S},\mathcal{T})$ then $\forall x\in A$ $\diam(\phi^n(A,x))\rightarrow 0$ when $n\rightarrow-\infty$. 
\end{itemize}
\end{thm}
\begin{proof} If  (a) is not true then there are $A\subset W^s_\eta(\mathcal{S},\mathcal{T})$, $x\in A$ and $\delta_0>0$ so that $\diam(\phi^{n_i}(A,x))\geq\delta_0$ for an increasing sequence of integers $\{n_i\}$. Thus, for each $i$, there is $y(i)\in A$ such that $d(\phi^{n_i}(x),y(i)_{n_i})\geq\frac{\delta_0}{2}$. 
Recall that $y(i)_{n_i} \in \phi^{n_i}(A,x) $.
There is no loss assuming 
$\phi^{n_i}(x)\rightarrow a\in T_j$ and $y(i)_{n_i}\rightarrow b\in S_j$ with $a\neq b$. Since $C(\bigcup_{i=1}^kS_i)$  is compact, \cite{N}, we can suppose $\lim_{n\rightarrow\infty}\phi^{n_i}(A,x)\rightarrow B\subset S_j$. Since $a,b\in B$, and $a\neq b$, $B$ does not reduce to a unique point and hence, for a fixed
 $k\in\mathbb{Z}$ we have that
\begin{eqnarray*}\diam(\phi^k(B,a))=\max_{c,d\in B} d(c_k^a,d_k^a))\,.
\end{eqnarray*}
Fix $c_k^a,d_k^a\in\phi^k(B,a)$. Thus, there are sequences of points $\{c(i)\}$ and $\{d(i)\}$ in $A$ such that $c(i)_{n_i}\rightarrow c$ and $d(i)_{n_i}\rightarrow d$. Since $\phi^{n_i+k}(x)\rightarrow\phi^k(a)$, \cite[lemma 2.9]{KS} implies that  $c(i)^x_{n_i+k}\rightarrow c^a_k$ and $d(i)^x_{n_i+k}\rightarrow d^a_k$. 
Furthermore, since for all $ i$  it holds
$$d(c(i)^x_{n_i+k},d(i)^x_{n_i+k})\leq \diam (\phi^{n_i+k}(A,x))\leq\eta\,,$$
we get $d(c^a_k,d^a_k)\leq\eta$ and as $c, d$ are arbitrary, we obtain
$$\diam(\phi^k(B,a))<\eta.$$  
This is true for every $k\in\mathbb{Z}$ and hence  $B\in W^s_\eta(\mathcal{S},\mathcal{T})\cap W^u_\eta(\mathcal{S},\mathcal{T})$. This leads to a contradiction because $X^t$ is $cw$-expansive and $B$ does not reduce to a point. 

Similarly we prove (b).

\end{proof}

Theorem \ref{var-estavel} motivates the following definition.

\begin{defi}\label{def-var-estavel-par}
The stable $W^s(\mathcal{S},\mathcal{T})$ and unstable $W^u(\mathcal{S},\mathcal{T})$ sets of a $\delta$-adequate pair of cross
sections $(\cS,\cT)$ are defined as
$$W^s(\mathcal{S},\mathcal{T})=\{A\in \cC(M); \exists\,\, x\in A\cap \bigcup^k_{i=1}T_i \mbox{ and } \diam(\phi^n(A,y)\stackrel{n\rightarrow+\infty}{\longrightarrow} 0, \ \forall y\in A \}$$
$$W^u(\mathcal{S},\mathcal{T})=\{A\in \cC(M); \exists \,\, x\in A\cap \bigcup^k_{i=1}T_i \mbox{ and } \diam(\phi^n(A,y)\stackrel{n\rightarrow-\infty}{\longrightarrow} 0, \ \forall y\in A \}.$$
\end{defi}
Here $\cC(M)$ is the set of all continuum subsets of $M$.

\section{Entropy for flows}
In this section we define a notion of topological entropy for 
$\delta$-adequate pairs of cross-sections and relate this notion with the topological entropy of the flow. 

Next we review the definition of topological entropy for flows given in \cite{B1,B2,B3}.
Given subsets $E,F\subset M$, we say that $E$ \emph{$(t,\delta)$-spans} $F$ if for all $x\in F$ there is $e\in E$ such that 
$$d(X^s(e)X^s(x))\leq\delta, \ \forall s\in[0,t].$$
Let $r_t(F,\delta)$ be the smallest cardinality of a $(t,\delta)$-spanning set for $F$. If $F$ is compact, then $r_t(F,\delta)$ is finite. Define
$$h(X^t|F,\delta)=\limsup_{s\rightarrow\infty}\frac{1}{t}\log r_s(F,\delta).$$
The \emph{topological entropy} of $X^t$ on $F$ is defined as
$$h(X^t|F)=\lim_{\delta\rightarrow 0}h(X^t|F,\delta).$$

Given $E,F\subset M$, we say that $E$ \emph{$(t,\gamma)$-weakly spans} $F$, if for every $x\in F$, there are $e\in E$ and $h\in \Hom(\mathbb{R},0)$ such that  for every $ s\in [0,t]$ it holds
$$d(X^{h(s)}(x),X^s(e))\leq\delta.$$
 Let $R_t(F,\gamma)$ be the smallest cardinality of a  $(t,\gamma)$-weakly spanning set for $F$, and define
$$H(X^t|F,\gamma)=\limsup_{t\rightarrow\infty}\frac{1}{t}\log R_t(F,\gamma).$$ 

Define $H(X^t|F)=\lim_{\delta\rightarrow 0}H(X^t|F,\gamma)$.
If $F=M$ we use the notation $H(X^t|M)=H(X^t)$ and $h(X^t|M)=h(X^t)$.

In \cite[theorem 10]{T}, it is proved that If $X^t$ is a flow without fixed points then $H(X^t)=h(X^t)$.

\begin{defi}\label{d-entropia}
We say $h(X^t)$ is the {\emph{topological entropy}} of the flow $X^t$.
\end{defi}

Next we shall define entropy for pairs $(\cS,\cT)$ of $\delta$-adequate cross sections. For this,
we start defining spanning and separate sets for $(\cS,\cT)$.

Fix a pair $(\mathcal{S},\mathcal{T})$ of $\delta$-adequate cross sections. Given $\gamma > 0$, we say that a set $E\subset\bigcup_{i=1}^kT_i$ is a \emph{$(n,\gamma)$-spanning} for $(\cS,\cT)$ if for all $x\in \bigcup_{i=1}^kT_i$ there exists $e\in E$ such that 
$$d(\phi^i(x),e^x_i)<\gamma\,\, , \forall \,\, i\in \{0,...,n\}.$$   
Let $r'(n,\gamma)$ be the smallest  cardinality of $(n,\gamma)$-spanning sets and define
\begin{equation}\label{entropia}
H'((\mathcal{S},\mathcal{T}),\gamma)=\limsup_{n\rightarrow\infty}\frac{1}{n}\log r'(n,\gamma),\quad H'((\mathcal{S},\mathcal{T}))=\lim_{\gamma\rightarrow 0}H'((\mathcal{S},\mathcal{T}),\gamma).
\end{equation}

We say that $E\subset\bigcup_{i=1}^kT_i$ is a \emph{$(n,\gamma)$-separated} if for
all $ x,y\in E$,  either $y_i^x$ is not defined for some $i\in\{0,...,n\}$ or there exists $i\in\{0,...,n\}$ such that  
$$d(\phi^i(x),y^x_i)\geq\gamma.$$   
Let $s'(n,\gamma)$ be the largest cardinality of  $(n,\gamma)$-separated sets and define
\begin{equation}\label{entropia-2}
s'((\mathcal{S},\mathcal{T}),\gamma)=\limsup_{n\rightarrow\infty}\frac{1}{n}\log s'(n,\gamma),\quad s'((\mathcal{S},\mathcal{T}))=\lim_{\gamma\rightarrow 0}s'((\mathcal{S},\mathcal{T}),\gamma).
\end{equation}

\begin{lemma} Let $X^t$ be a without fixed point flow and $(\mathcal{S},\mathcal{T})$ 
be a pair of $\delta$-adequate cross sections for $X^ t$.  Then for all $\gamma >0$ we have  $R'(n,\gamma)\leq s'(n,\gamma)\leq R'(n,\frac{\gamma}{2})$.
\end{lemma}
\begin{proof} Let $E$ be a largest $(n,\gamma)$-separated set and suppose that is not a $(n,\gamma)$-spanning. Thus exists $x\in\bigcup_{i=1}^kT_i$ such that if $e\in E$ then there is $i\in\{0,...,n\}$  
$$d(\phi^i(x),e_i^x)\geq\gamma.$$  
then, $E\cup\{x\}$ is a $(n,\gamma)$-separated set, leading to a contradiction. Then $$R'(n,\gamma)\leq s'(n,\gamma).$$
To prove the other inequality, let $E$ be a $(n,\gamma)$-separated and $F$ be a $(n,\frac{\gamma}{2})$-spaning set. Then $\forall x\in\bigcup_{i=1}^kT_i$ we can take $g(x)\in F$ such that 
$$d(\phi^i(x),g(x)^x_i)<\frac{\gamma}{2}$$
$\forall i\in\{0,...,n\}$. If $g(x)=g(y)$ $x,y,g(x)\in T_j$ for some $j\in\{1,...,k\}$, and $g(x)_i^x=g(y)_i^y$ $\forall i\in\{0,...,n\}$. The last statement is easy to proof by induction: if for $i_0<n$ is true that $\forall i\in\{0,...,i_0\}$  $g(x)_i^x=g(y)_i^y$ then if $g(x)_{i_0+1}^x\neq g(y)_{i_0+1}^y$ ou $x_{i_0+1}^y$ is not defined or $y_{i_0+1}^x$ is not defined, absurd, because
$$d(\phi^{i_0}(x),\phi^{i_0}(y))\leq d(\phi^{i_0}(x),g(x)^{i_0}_x)+d(\phi^{i_0}(y),g(y)^{i_0}_y)< \gamma.$$
Thus, $\forall i\in\{0,...,n\}$ we have
\begin{eqnarray*} d(\phi^i(x),\phi^i(y))\leq d(\phi^i(x),g(x)_i^x)+d(\phi^i(y),g(x)_i^y)<\gamma.
\end{eqnarray*} 
 Since $E$ is $(n,\frac{\gamma}2)$-separated and $g$ restricted to $E$ is injective we obtain that  $\#E\leq\#F$ and $s'(n,\gamma)\leq R'(n,\frac{\gamma}{2})$, finishing the proof. 
\end{proof}

\begin{clly} $H'((\mathcal{S},\mathcal{T}))=s'((\mathcal{S},\mathcal{T}))$.
\end{clly}

\begin{defi}
$H'((\mathcal{S},\mathcal{T}))$ is, by definition,  the \emph{topological entropy} 
of the $\delta$-adequate
pair  of  cross sections $(\mathcal{S},\mathcal{T})$.
\end{defi}

\begin{thm}\label{t-entropia} Let $X^t$ be a flow without fixed points and $(\mathcal{S},\mathcal{T})$ be a $\delta$-adequate pair of cross sections. If $H'(\mathcal{S},\mathcal{T})>0$ then $h(X^t)>0$.
\end{thm}
\begin{proof} Fix $\delta>0$ such that $s'((\cS,\cT),\delta)>\frac{H'(\cS,\cT)}{2}>0$. Fix $n\in\mathbb{N}$, consider $E\subset \bigcup_{i=1}^k T_i$ maximal $(n,\gamma)$-separated set for $(\mathcal{S},\mathcal{T})$. Since $X^{[0,n\epsilon]}(x)$ cross at least $n$ cross-sections $T_i's$ for each $x$, there is $\gamma>\gamma'>0$ such that $E$ is a $(n\epsilon,\gamma')$-separated set for the flow $X^t$. Thus, $s'(n\epsilon,\gamma',(\mathcal{S},\cT)\leq s(n\epsilon,\gamma,X^t)$. 
Here $s'(n\epsilon,\gamma',(\mathcal{S},\cT)$ is the maximal cardinality of
$(n\epsilon,\gamma)$-separated sets of $(\cS,\cT)$ and $s(n\epsilon,\gamma,X^t)$
is the maximal cardinality of $(n\epsilon,\gamma)$-separated sets of $X^t$.
Then,
\begin{eqnarray*} s(\gamma,X^t)&=&\limsup_{u\rightarrow\infty}s(u,\gamma,X^t) \\
&\geq& \limsup_{n\rightarrow\infty}s(n\epsilon,\gamma,X^t) \\
&\geq& \limsup_{n\rightarrow\infty}s'(n\epsilon,\gamma,(\cS,\cT)) \\
&=& s'(\gamma,(\cS,\cT)) \\
&>& \frac{H'((\cS,\cT))}{2}.
\end{eqnarray*}
Therefore, $h(X^t)=\lim_{\gamma\rightarrow 0} s(\gamma,X^t)>0$.
\end{proof}

\section{Entropy for $cw$-expansive flows}
In this section we prove our main theorem. First we need to prove some technicals results.

\begin{lemma}\label{Uniforme} Let $X^t$ be a $cw$-expansive flow with $\eta$ as expansive constant. For all $\epsilon_0\in (0,\frac{\eta}{2}]$, there is $\delta_0>0$ such that if 
$A\in \cC( \bigcup_{i=1}^kS_i)$, $x\in A\cap\bigcup_{i=1}^kT_i$, $\diam(A)\leq\delta_0$ and if
there is $n>0$ such that
$$\sup\{\diam\phi^i(A,x);i=1,...,n\}\in[\epsilon_0,2\epsilon_0]$$    
then $\diam(\phi^n(A,x))\geq\delta_0.$
\end{lemma}

\begin{proof} The proof goes by contradiction.  If the thesis fails, there are a sequence $\{A_i\}$ of connected and compact sets in $\bigcup^k_{i=1}S_i$,  a sequence of points $\{x_i\}$  in $A_j\cup\bigcup_{i=1}^kS_i$ and an increasing sequence $\{n_i\}$ of natural numbers such that:
\begin{enumerate}
\item $\diam A_i<\frac{1}{i}$, $\forall i\in\mathbb{N}$;
\item $\sup\{\diam\phi^i(A_i,x_i); i=1,...,n_i \}\in[\epsilon_0,2\epsilon_0]$;
\item $\diam (\phi^{n_i}(A_i,x_i))<\frac{1}{i}\,.$ 
\end{enumerate}
By $(1)$, for each $i$ we can choose $0<m_i<n_i$ so that $\diam(\phi^{m_i}(A,x_i))\in[\epsilon_0,2\epsilon_0]$. We will prove that
\begin{eqnarray}\label{1} \lim_{i \to \infty} m_i=\lim_{i\to \infty} (n_i-m_i)=\infty.
\end{eqnarray}  
Fix $N>0$. By continuity, there is $\delta_1>0$ such that if $x\in\bigcup^k_{i=1}T_i$ and $y\in\bigcup^k_{i=1}S_i$ satisfy $d(x,y)<\delta_1$ then $d(\phi^n(x),y^x_n)<\epsilon_0$ for all $n\in\{-N,...,-1,0,1,...,N\}$. Take $i_0$ a natural number sufficiently large such that $\frac{1}{i_0}<\delta_1$. Then for all $i>i_0$ we have that $\diam(A_i)<\delta_1$ and $\diam(\phi^{n_i}(A_i,x_i)<\delta_1$. Therefore,
$$\diam(\phi^{j}(A_i,x_i)<\epsilon_0 \ \mbox{and } \diam(\phi^{n_i-j}(A_i,x_i)<\epsilon_0, \ \forall j\in\{1,...,N\}.$$ 
Thus,  $m_i)>N$ and $n_i-m_i>N$, and since that take $N$ arbitrary, we proved \ref{1}.     

We can assume $\lim \phi^{m_i}(A_i,x_i)\rightarrow B$ and $\phi^{m_i}(x_i)\rightarrow x\in B$. Since $\diam(B) \geq \epsilon$ we have that $B$  is not degenerate. Fix an
arbitrary integer  $n$. Then for $i$ sufficiently big we have $m_i+n\in\{0,...,n_i\}$. Therefore, for  $i$ sufficiently big, we have $\diam(\phi^{m_i+n}(A_i,x_i))<2\epsilon_0$. So, by \cite[Lemma 2.9]{KS}, we obtain 
$$\phi^{m_i+n}(A_i,x_i)\rightarrow \phi^{n}(B,x).$$  
Since  $n$ is arbitrary,  we obtain that $\phi^{n}(B,x)<2\epsilon_0<\eta$ for all $n\in\mathbb{Z}$ which leads to a contradiction to the fact of $X^t$ be $cw$-expansive.
\end{proof}

\begin{lemma}\label{l2} Let $X^t$ be a $cw$-expansive flow and $\epsilon_0$ and $\delta_0$ as the previous lemma. If $A\in \cC(\bigcup_{i=1}^kS_i)$, $A\cap\bigcup_{i=1}^kT_i\neq\emptyset$ such that $\diam A<\delta_0$ and for some integer $m$ and  $x\in A$ $\diam \phi^m(A)\geq\epsilon_0$. Then,
\begin{enumerate}
\item If $m\geq 0$ then $\diam\phi^n(A,x)\geq\delta_0$ for all $n\geq m$. More precisely, there is a continuum $B\subset A$, with $x\in B$ such that $$\sup\{\diam\varphi^j(B,x);j=1,...,n\}\leq\epsilon_0$$ e $\diam\phi^n(B,x)=\delta_0$;  
\item If $m< 0$ then $\diam\phi^{-n}(A,x)\geq\delta_0$ for all $n\geq -m$. More precisely, there is a continuum $B\subset A$, with $x\in B$ such that 
$$\sup\{\diam\phi^{-j}(B,x);j=1,...,n\}\leq\epsilon_0\quad \mbox{and}\quad \diam\phi^{-n}(B,x)=\delta_0.$$ 
\end{enumerate}
\end{lemma}
\begin{proof} Let  $m\geq 0$. By \cite{N}, there is an arc $c:[0,1]\rightarrow \cC(\bigcup_{i=1}^{k}S_i)$ from $\{x\}$ to $A$ such that if $r\leq s$ then $c(r)\subset c(s)$. Fix $n\geq m$. Define a map $F:[0,1]\rightarrow [0,\infty]$ by
$$F(r)=\sup\{\diam\phi^j(c(r),x);j=0,1,...,n\}. $$  
Take $r_0\in[0,1]$ such that $r_0\in F^{-1}(\epsilon_0)$. By Lemma \ref{Uniforme}, 
$$\diam\phi^n(A,x)\geq\diam\phi^n(c(r_0))\geq\delta_0.$$
Define $D:C(c(r_0))\rightarrow[0,\infty]$ by $D(C)=\diam\phi^n(C)$. Since $C(r(a_0))$ is connected, $D^{-1}(\delta_0)$ is non-empty and therefore, all $B\in D^{-1}(\delta_0)$ satisfies $\diam\phi^n(B,x)=\delta_0$ and $$\sup\{\diam\phi^{j}(B,x);j=1,...,n\}\leq\epsilon_0.$$
Similarly we prove for $m \leq 0$.
 
\end{proof}

\begin{clly}\label{c1} Let $X^t$ be a $cw$-expansive flow, $\delta_0$ and $\epsilon_0$ as in  Lemma \ref{Uniforme}. Then, for each $\gamma>0$ there exists $N>0$ such that if $A\subset \bigcup_{i=1}^kS_i$ is a continuum with $A\cap\bigcup_{i=1}^kT_i\neq\emptyset$ and $\diam A\geq\gamma$, then for all $x\in A\cap\bigcup_{i=1}^kT_i$  either $\diam\phi^n(A,x)\geq\delta_0$ for all $n\geq N$ or $\diam\phi^{-n}(A,x)\geq\delta_0$ for all $n\geq N$.   
\end{clly}

\begin{lemma}\label{l.6} Let $X^t$ be a $cw$-expansive flow and $(\mathcal{S},\mathcal{T})$ be a $\delta$-adequate pair of cross sections. If there is a cross section $T_i$ with topological dimension greater than zero, then there exists a non degenerated continuum $A\subset S_i$ such that either $A\in W^s(\mathcal{S},\mathcal{T})$ or $A\in W^u(\mathcal{S},\mathcal{T})$. 
\end{lemma}

\begin{proof} Let $C\subset T_{i_0}$ be a continuum non degenerated, with $\diam C\leq \delta_0$. Suppose that any $C'\subset C$ continuum in $C$ does not belong to $W^s_\epsilon(\mathcal{S},\mathcal{T})$. Choose a sequence of connected and compact sets $$C_1\subset C_2\subset...\subset C_n\subset...,$$
 and a sequence of natural numbers 
$n(1)<n(2)<...$, such that $\diam C_i\rightarrow 0$, $C_i\rightarrow x$, $x\in C_i$ for all $i$,
$$\sup \{\diam\phi^{j}(C_i);j=0,1,...,n(i)\}\leq\epsilon_0$$ 
and $\diam \phi^{n(i)}(C_i)\leq\delta_0$ for all $i$. We can suppose that $\phi^{n(i)}(x)\rightarrow x_0\in T_{j_0}$ and $\phi^{n(i)}(C_i)\rightarrow A\in S_{j_0}$ for some $j_0\in \{1,...,k\}$.  So $\diam A\geq\delta_0$. 
Given $n\in \mathbb{N}$, \cite[Lemma 2.9]{KS} implies that $\phi^{n(i)-n}(C_i,x)\rightarrow \phi^{-n}(A_i,x_0)$, and since $\diam \phi^{n(i)-n}(C_i,x)<\epsilon_0$ for all $i$, $\diam \phi^{-n}(A,x_0)<\epsilon_0$. Since $n$ is arbitrary, we obtain $A\in W^u_\epsilon(\mathcal{S},\mathcal{T})$.

\end{proof}

Next we give a sketch of the proof of Theorem \ref{main}.
Fix $\delta_1>0$ small. Given a $\delta$-adequated pair $(\cS,\cT)$ of cross sections, we want to exhibit  $l\in\mathbb{N}$ (independent of $\delta_1$) such that for each $m\in\mathbb{N}$ there is a $(ml,\frac{\delta_1}{3})$-separated set $E$ for $(\cS,\cT)$,  with $2^m$ points.
Since the topological dimension of $M$ is greater than $1$, there is a cross-section $T_i$ with topological dimension greater than $0$. Lemma \ref{l.6} implies that there is a compact and connected set 
$A\in W^u_\eta(\mathcal{S},\mathcal{T})$, $A \subset T_i$, and by Lemma \ref{l2} we can assume that  $\diam A=\frac{\delta_1}{3}$. Fix a natural number $m$. By Corollary \ref{c1} and Lemma \ref{l2} we can find two connected and compact sets $A_1$ and $A_0$ in $\phi^N(A,x)$ such that $d(A_1,A_0)\geq\frac{\delta_1}{3}$, and $\diam(A_{i_1})=\frac{\delta_1}{3}$, for $i_1=0,1$. 
\\
We would like to repeat this process with $A_1$ and $A_2$. If so, we would find a finite collection 
$\{(A_{i_1,i_2,...,i_j},a_{i_1,i_2,...,i_j})\}$ with $i_k\in\{0,1\}$, and $j\leq m$, such that:
\begin{itemize}
\item $a_{i_1,...,i_k}\in A_{i_1,...,i_k}$ with $A_{i_1,...,i_k}$  compact, connected  subsets of $\phi^N(A_{i_1,...,i_{k-1}},a_{i_1,...,i_{k-1}})$ satisfying
$$\diam(A_{i_1,...,i_k})=\frac{\delta_1}{3}\quad \text{ and }\quad  d(A_{i_1,...,i_{k-1},0},A_{i_1,...,i_{k-1},1})\geq\frac{\delta_1}{3};$$ 
\item $A_{i_1,...,i_j}\in W^u_\eta(\mathcal{S},\mathcal{T})$.
\end{itemize}
Thus, for  each $A_{i_1,...,i_m}$ we could choose a point $b(i_1,...,i_m)\in A_{i_1}$ such that 
\begin{itemize}
\item $c_i(1)=b(i_1,...,i_m)^{a_{i_1}}_{N}\in A_{i_1,i_2}$;   
\item $c_i(j)=c(j-1)^{a_{i_1,...,i_{j}}}_{N}\in A_{i_1,...,i_{j+1}}$ for $j\leq m-1$.
\end{itemize}
Then the $E= \{ b_{i_1,...,i_m}, i_k \in \{0,1\}, 1\leq k \leq m\}$ is $(mN,\frac{\delta_1}{3})$-separated for the $\delta$-adequate pair of cross-sections $(\{T\},\{S\})$. In this case, $l=N$.
\\
But it is not always possible to repeat the above argument. In fact, for $A_{i_1}\subset (\bigcup S_j)\setminus(\bigcup T_j)$ there are no points in $ A_{i_1}\cap\bigcup T_j$, and so $A_{i_1}\cap \domi(\phi)=\emptyset$. 

\begin{figure}[htb]
\begin{center}
\psfrag{a}{$T^1$}
\psfrag{b}{$S^1$}
\psfrag{c}{$T^2$}
\psfrag{d}{$S^2$}
\psfrag{e}{$S^3$}
\psfrag{f}{$T^3$}
\includegraphics[height=4cm]{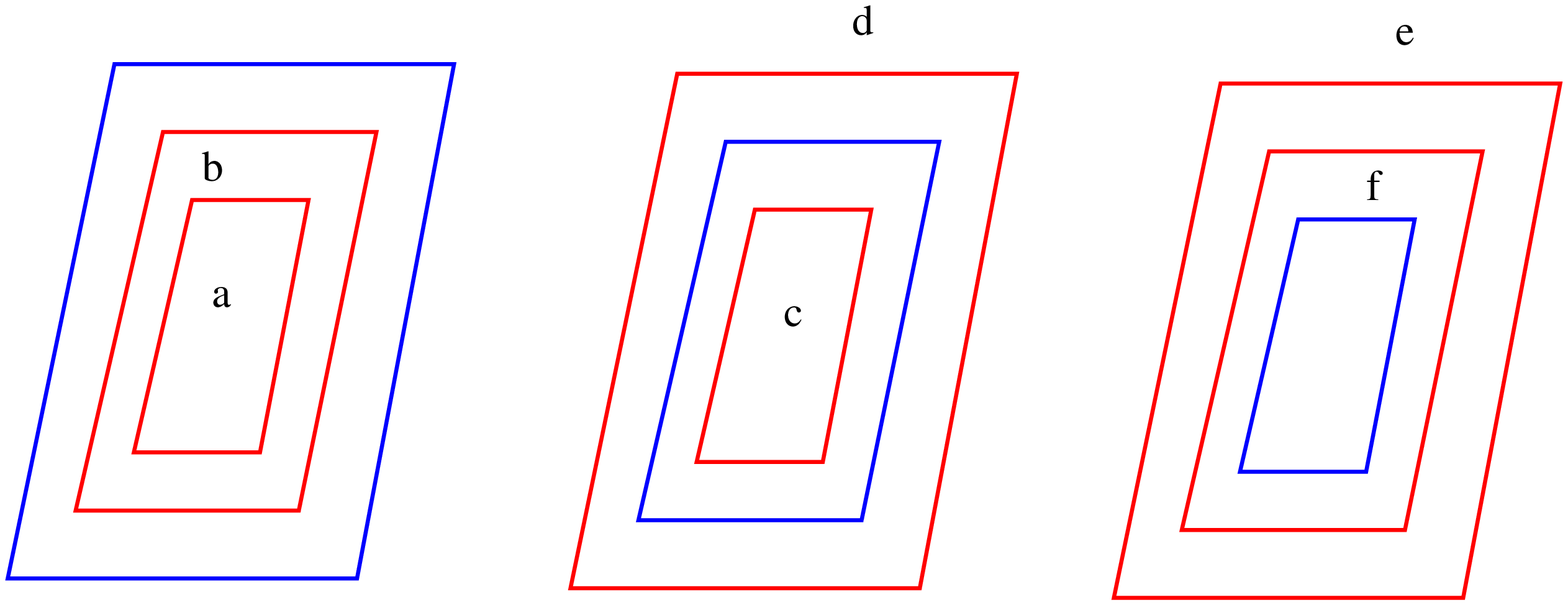}
\caption{} \label{fig5}
\end{center}
\end{figure}

To bypass this difficulty,   we reason with $3$ pairs of $\delta$-adequate families of cross-sections: 
$$(\{T^1_i\},\{S^1_i\}), (\{T^2_i\},\{S^2_i\}), (\{T^3_i\},\{S^3_i\}),$$
such that for all $i\in\{0,...,k\}$, $T^1_i=T^2_i$, $T^3_i=S^1_i$, $S^2_i=S^3_i$, (observe $T^1_i\subset T^2_i\subset S^3_i$).

\begin{figure}[htb]
\begin{center}
\psfrag{a}{$x$}
\psfrag{b}{$A$}
\psfrag{c}{$\phi_3(A,x)$}
\psfrag{d}{$\phi_2(A,x)$}
\includegraphics[height=4cm]{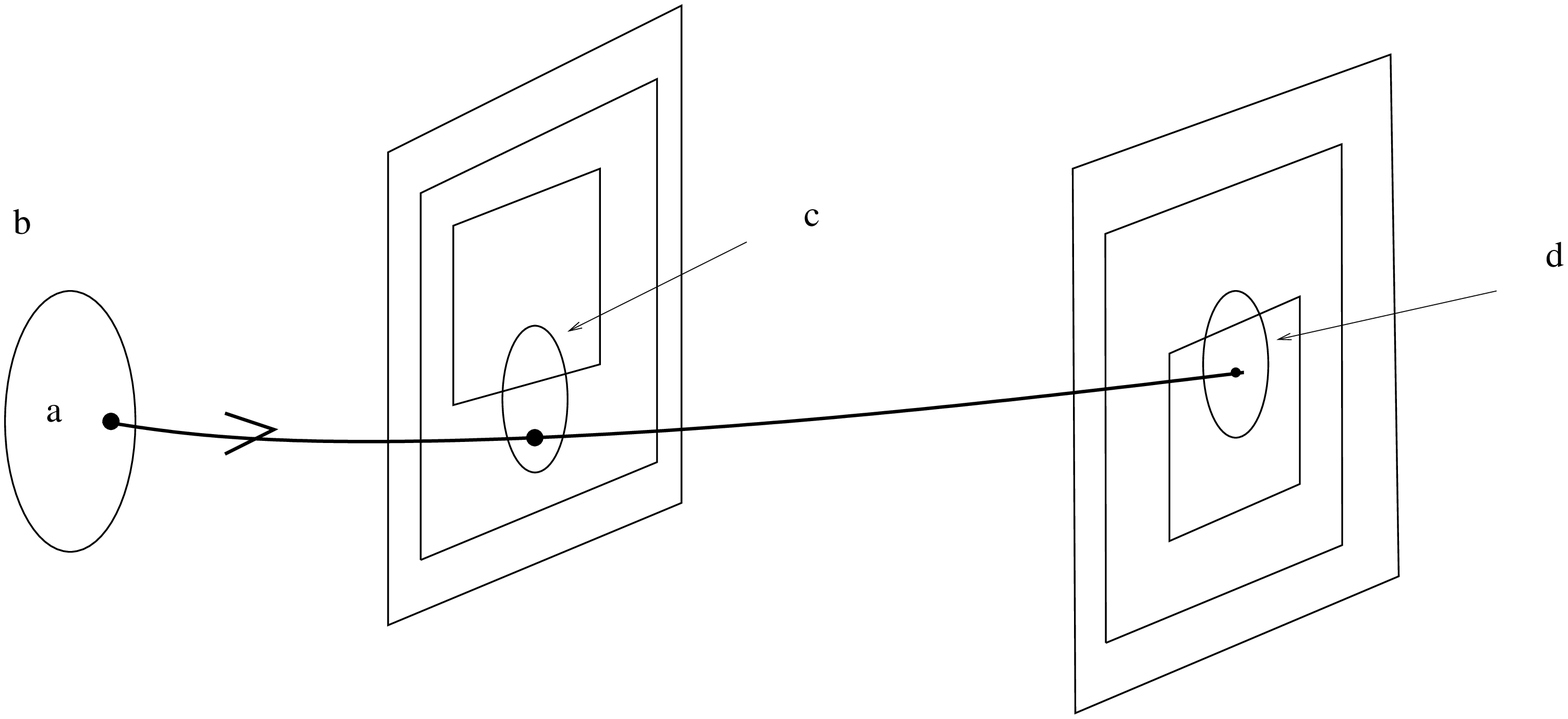}
\caption{} \label{fig6}
\end{center}
\end{figure}

This strategy leads to consider, besides the first return map $\phi$ above,  three other first return maps $\phi_1$, $\phi_2$ and $\phi_3$    relative to the distinct pairs of $\delta$-adequate cross sections $(\cS^1,\cT^1)$, $(\cS^2,\cT^2)$ and $(\cS^3,\cT^3)$.
Further, we shall define an auxiliary map $\varphi: \bigcup T_j^3 \to \bigcup T_j^1$ such that even
if $A_{i_1}\cap \domi(\phi_1)=\emptyset$ we get that $A_{i_1} \subset \phi_1(A,x)
\subset \domi(\varphi)$, and this allow us to continue with the argument.

Since $(\mathcal{S}^i,\mathcal{T}^i)$, $i=1, 2, 3$ are finite families of compact sets there exists $k\in\mathbb{N}$ such that for each $z\in (\bigcup S^2_j)$ there is  $r\in \{1,...,k\}$ with $\varphi(A,z)=\phi_3^r(A,z)$. Therefore, reasoning as before we find a $(lm,\frac{\delta_1}{3})$-separated set  for $(\mathcal{S}^3,\mathcal{T}^3)$, and in this case we put $l=kN$.  Next we formalize this in the proof of
our main result.
\vspace{0.2cm}

\begin{maintheorem} \label{main}If $X^t$ is a $cw$-expansive flow in compact metric space $M$ with topological dimension greater than $1$ then the entropy $h(X^t)$ of the flow is positive.
\end{maintheorem} 

\noindent{\it Proof.} Take $3$ pairs of $\delta$-adequate families of cross-sections
$(\cT^j, \cS^j)$, $1\leq j \leq 3$, 
$$
(\cT^j, \cS^j)= \{ (\{T^j_i\}, \{S^j_i\}); T^j_i \in \cT^j, S^j_i \in \cS^j\}
$$
such that for all $i\in\{0,...,k\}$, $T^1_i=T^2_i$, $T^3_i=S^2_i$, $S^1_i=S^3_i$, (observe $T^1_i\subset S^2_i\subset S^1_i$). 

For $x\in\bigcup_{i=1}^k S^2_i$ define $\varphi(x)=X^t(x)$ where $t>0$ is the smallest positive time such that $X^t(x)\in \bigcup_{i=1}^k T^1_i$. Observe that $t\in[\theta_3,\epsilon]$ (where $\theta_3$ 
is defined as in (\ref{e.theta}) relative to the third pair of $\delta$-adequate cross-sections families). Let $\theta'$ be the smallest positive time  $t$ such that any point $a\in\bigcup_{i=1}^kS_i^2$ satisfies $X^{(0,t)}(a)\cap\bigcup_{i=1}^kT_i^1\neq\emptyset$. 
Note that the segment of orbit $X^{[0,\theta_3]}(x)$ intersects cross sections of type $T^3_i$
in at most $[\frac{\epsilon}{\theta'}]$ number of times. Hence
$\varphi(x)=\phi_1^i(x)$ for some $i\in\{1,...,[\frac{\epsilon}{\delta'}]\}$ (where $\phi_1$ is the 
first return map 
associated to the pair of families $(\cT^1,\cS^1)$ of cross-sections).

 Let $\epsilon_1\in(0,\epsilon_0)$ be such that if $x,y\in S^2_i$ and $d(x,y)<\epsilon_1$ and $t$ is a real number with $|t|\leq[\frac{\epsilon}{\delta'}]$ and $X^t(x)\in T^1_j$, then $X^t(y)\in D^j_\rho$. If $x\in T_i^1$ and $y\in S^2_i$ with $d(x,y)<\epsilon_1$ define a set of points $\{y^x_{\varphi,0},...,y^x_{\varphi,n}\}\subset \cO(y)$ with $y^x_{\varphi,0}=y$ and $y^ x_{\varphi,l}=P_\rho^l(X^t(y^x_{\varphi,l-1}))$, where $t>0$ is the smallest positive time such that $\varphi^l(x)=X^t(\varphi^{l-1}(x))$, and $l$ is such that $\varphi^j(x)\in T_l$. 
Recall that $P^l_\rho$ is the projection defined in (\ref{e.projecao}).
  We can argue as above whenever $d(\varphi^j(x),y_j)<\epsilon_1$. Similarly  for $j<0$. 
  
  Take $\delta_1\in(0,\delta_0)$ such that if $x\in \bigcup_{i=1}^kT_i^3$ and $y\in \bigcup_{i=1}^kS_i^3$ then $d(\phi_3^i(x),y_i^x)<\epsilon_1$ for all $i\in\{1,...,[\frac{\epsilon}{\delta'}]\}$.

By Lemma \ref{l.6} we can suppose that there exists a non degenerated compact and connected set $A\in W^u_\eta(\mathcal{S}^2,\mathcal{T}^2)$ in some cross section $T^1_i$, and by Lemma \ref{l2} we can assume that $\diam A=\frac{\delta_1}{3}$.  
Let $N$ be given at Corollary \ref{c1}.
Hence if $D\in\bigcup_{i=1}^kS^1_i$ is a compact and connected set with $D\cap\bigcup_{i=1}^kT^1_i\neq\emptyset$,  $\diam D\geq\frac{\delta_1}{3}$, then for all $x\in D$ we have $\max\{\diam\phi_1^j(D,x); |j|\leq N\}>\eta$.
 Hence, by Lemma \ref{l2} and the definition of $\varphi$, if $D\in W^u_\eta(\mathcal{S}^2,\mathcal{T}^2)$ and $\diam D=\frac{\delta_1}{3}$, we get that for all $x\in D$ it holds $\diam\varphi^N(D,x)\geq\delta_1$. 
 
 Fix $m\in \NN$ and $x\in A$.  
 As $\diam\varphi^N(D,x)\geq\delta_1$, 
Lemma \ref{l2} implies that there are two connected and compact sets $A_1$ and $A_0$ in $\varphi^N(A,x)$ such that $d(A_1,A_0)\geq\frac{\delta_1}{3}$ and $\diam(A_{i_1})=\frac{\delta_1}{3}$, for $i_1=0,1$.
Choosing $a_{i_1}\in A_{i_1}$ we get that $\diam\varphi^N(A_{i_1},a_{i_1})\geq\delta_1$
and so Lemma \ref{l2} implies that there are
sets $A_{i_1,0}$ and $A_{i_1,1}$ contained in $\varphi^N(A_{i_1},a_{i_1})$ such that $\diam(A_{i_1,i_2})=\frac{\delta_1}{3}$ and $d(A_{i_1,0},A_{i_1,1})\geq\frac{\delta_1}{3}$, with $i_1,i_2\in \{0,1\}$. 

Continuing this construction, we find a finite collection of sets $\{(A_{i_1,i_2,...,i_j},a_{i_1,i_2,...,i_j})\}$ with $i_k\in\{0,1\}$,  $1\leq j\leq m$, such that:
\begin{itemize}
\item $A_{i_1,...,i_k}$, are compact and connected  subsets of $\varphi^N(A_{i_1,...,i_{k-1}},a_{i_1,...,i_{k-1}})$ with 
$$\diam(A_{i_1,...,i_k})=\frac{\delta_1}{3} \text{ and } d(A_{i_1,...,i_{k-1},0},A_{i_1,...,i_{k-1},1})\geq\frac{\delta_1}{3};$$ 
\item $A_{i_1,...,i_j}\in W^u_\eta(\mathcal{S}^2,\mathcal{T}^2)$.
\end{itemize}
For each $A_{i_1,...,i_m}$ pick a point $b(i_1,...,i_m)\in A_{i_1}$ such that 
\begin{itemize}
\item $c_i(1)=b(i_1,...,i_m)^{a_{i_1}}_{\varphi,N}\in A_{i_1,i_2}$;   
\item $c_i(j)=c(j-1)^{a_{i_1,...,i_{j}}}_{\varphi,N}\in A_{i_1,...,i_{j+1}}$ for $j\leq m-1$.
\end{itemize}
\begin{claim}\label{c-1} The set $E$ consisting of all points $b_{i_1,...,i_m}$ as above is $(m[\frac{\epsilon}{\theta'}]N,\frac{\delta_1}{3})$-separated for the $\delta$-adequate pair of families of cross-sections $(\{T^3_i\},\{S^3_i\})$.
\end{claim}  
Proof. 
Let $b_{i_1,...,i_m}\neq b_{l_1,...,l_m}$ and $k$ the smallest natural number such that $i_k\neq l_k$. 
Then $c_l(k-1)\in A_{i_1,...,i_{k-1},l_k}$ and $c_i(k-1)\in A_{i_1,...,i_{k-1},i_k}$ which implies 
$$d(c_i(k-1),c_l(k-1))\geq\frac{\delta_1}{3}.$$
Since $c_l(k-1)=\phi_3^M(b_{l_1,...,l_m})$ and $c_i(k-1)=\phi_3^M(b_{i_1,...,i_m})$, for some $M\in\{1,...,[\frac{\epsilon}{\theta'}]Nk\}\subset \{1,...,[\frac{\epsilon}{\theta'}]Nm\}$ the claim follows.

Now, Claim \ref{c-1} implies that $s'(mN\frac{\epsilon}{\theta'})\geq 2^m$ and so
\begin{eqnarray*} s'(\{\mathcal{T}_3,\mathcal{S}_3\},\frac{\delta_1}{3})&=&\limsup_{n\rightarrow\infty}\frac{1}{n}log( s'(n,\frac{\delta_1}{3})) \\
&\geq& \limsup_{m\rightarrow\infty}\frac{1}{mN[\frac{\epsilon}{\theta'}]+1}log (s'(mN[\frac{\epsilon}{\theta'}]+1,\frac{\delta_1}{3})) \\
&\geq& \limsup_{m\rightarrow\infty}\frac{m}{mN[\frac{\epsilon}{\theta'}]+1}log(2) = \frac{1}{N[\frac{\epsilon}{\theta'}]}log(2)>0. 
\end{eqnarray*}
Thus $H'(\{\mathcal{T}_3,\mathcal{S}_3\})=s'(\{\mathcal{T}_3,\mathcal{S}_3\})>0$ and 
 Theorem \ref{t-entropia}  implies that $h(X^t)>0$, finishing the proof.

\hfill{$\square$}

\vspace{1cm}
\noindent
{\em A. Arbieto,W. Cordeiro and  M. J. Pacifico }
\vspace{0.2cm}

\noindent Instituto de Matem\'atica,
Universidade Federal do Rio de Janeiro,
C. P. 68.530\\ CEP 21.945-970,
Rio de Janeiro, RJ, Brazil.
\vspace{0.2cm}

\noindent E-mail: arbieto@im.ufrj.br \hspace{.5cm}  welingtonscordeiro@gmail.com \hspace{.5cm} pacifico@im.ufrj.br

\end{document}